\documentclass[12pt]{amsart}
\pdfoutput=1
\usepackage{graphicx}
\usepackage{amssymb}
\usepackage{amsmath}
\usepackage{amsthm}
\usepackage{amscd}
\usepackage{epsfig}
\allowdisplaybreaks[4]
\newtheorem{theorem}{Theorem}[section]

\newtheorem{corollary}[theorem]{Corollary}
\newtheorem{lemma}[theorem]{Lemma}
\newtheorem{example}[theorem]{Example}
\newtheorem{proposition}[theorem]{Proposition}

\newtheorem{question}[theorem]{Question}
\newtheorem{remark}[theorem]{Remark}
\theoremstyle{definition}
\newtheorem{definition}[theorem]{Definition}
\numberwithin{equation}{section}

\newcommand{\Z}{\mathbb Z}
\newcommand{\N}{\mathbb N}

\newcommand{\U}{\mathcal{U}}
\newcommand{\cal}{\mathcal}

\begin{document}

\title{Entropy dimension of measure preserving systems}

\author [Dou Dou, Wen Huang and Kyewon Koh Park]{Dou Dou, Wen Huang and Kyewon Koh Park}

\address{Department of Mathematics, Nanjing University,
Nanjing, Jiangsu, 210093, P.R. China} \email{doumath@163.com}

\address{Department of Mathematics, University of Science and Technology of China,
Hefei, Anhui, 230026, P.R. China} \email{wenh@mail.ustc.edu.cn}

\address{Center for Mathematical Challenges, Korea Institute for Advanced Study, Seoul 130-722, Korea}
\email{kkpark@kias.re.kr}

\thanks{}


\begin{abstract}
The notion of  metric entropy dimension is introduced to measure
the complexity of entropy zero dynamical systems.  For measure
preserving systems, we define entropy dimension via the dimension
of entropy generating sequences.  This combinatorial approach
provides us with a new insight to analyze the entropy zero
systems. We also define the dimension set of a system to
investigate the structure of the randomness of the factors of a system. The notion of a uniform
dimension in the class of entropy zero systems is introduced as a
generalization of a K-system in the case of positive entropy. We investigate the joinings among
entropy zero systems and prove the disjointness property among
some classes of entropy zero systems using the dimension sets.  Given a topological system, we compare topological entropy dimension with metric entropy dimension.
\end{abstract}

\maketitle

\section{Introduction}
Since entropy was introduced by Kolmogorov from information theory,
it has played an important role in the study of dynamical systems.
Entropy measures the chaoticity or unpredictability of a system. It
is well known as a complete invariant for the Bernoulli automorphism
class.  Properties of positive entropy systems have been studied in
many different respects along with their applications.  Comparing
with positive entropy systems, we have much less understanding and
less tools for entropy zero systems.  Entropy zero systems which are
called deterministic systems in the case of $\Z$-actions cover a
wide class of dynamical systems exhibiting different ``random''
behaviors or different level of complexities. They range from
irrational rotations on a circle, more generally isometry on a
compact metric space, Toeplitz systems to horocycle flows. Also many
of the physical systems studied recently show intermittent or weakly
chaotic behavior \cite{PM,ZE}. They have the property that a generic
orbit has sequences of 0's with density 1 and hence we would say
that they have very low complexity or randomness.  They do not have
finite invariant measures which are physically meaningful. Hence to
analyze the complexity of these systems, the notion of algorithemic
information content or Kolmogorov complexity has been employed
instead of the entropy. It measures the information content of
generic orbits of the system.

Many of general group actions like $\Z^n-$actions with
entropy zero have interesting subdynamics.  They exhibit diverse
complexities and their non-cocompact subgroup actions show very
different behavior \cite{BLind,Milnor,Park1,Park2}.  We may mention
a few known examples of entropy zero with their properties in the
case of $\Z^2-$actions:

\begin{enumerate}

\item
\begin{enumerate}
\item $h(\sigma^{(p,q)})= 0$ for $\forall(p,q)\in\mathbb{Z}^2$,
\item for any given $0<\alpha<2$,
$$\displaystyle\limsup_{n\rightarrow\infty}\frac{1}{n^{\alpha}}{H}(\bigvee_{(i,j)\in
R_n}\sigma^{-(i,j)}{P})>0;$$
\end{enumerate}

\item
\begin{enumerate}
\item $h(\sigma^{(p,q)})=0$ for $\forall (p,q)\in\mathbb{Z}^2$,
\item
$$\displaystyle\lim_{n\rightarrow\infty}\frac{1}{n}{H}(\bigvee_{(i,j)\in
R_n}\sigma^{-(i,j)}{P})>0;$$
\end{enumerate}

\item
\begin{enumerate}
\item $h(\sigma^{(1,0)})>0$ \item $h(\sigma^{(p,q)})=0$ for
$\forall(p,q)\neq(n,0)$,
\item
$$\displaystyle\lim_{n\rightarrow\infty}\frac{1}{n}{H}(\bigvee_{(i,j)\in
R_n}\sigma^{-(i,j)}{P})>0;$$
\end{enumerate}
\end{enumerate}
where $\{\sigma^{(p,q)}\}_{(p,q)\in \Z^2}$ is the $\Z^2$-action, $h(\sigma^{(p,q)})$ is the entropy of the single transformation $\sigma^{(p,q)}$, $R_n$ denotes the square of size $n\times n$ in $\Z^2$ and $P$ is some finite measurable partition.

The first example is by Katok and Thouvenot \cite{KT} and the second
one is constructed in \cite{Park2}.  Although the third example is
not written anywhere explicitly, it is known that the example is by
Ornstein and Weiss, also independently by Thouvenot.

In his study of Cellular Automaton maps \cite{Milnor}, J. Milnor
considered the Cellular Automaton maps together with horizontal
shifts as $\Z^2$-actions of zero entropy.  He introduced the notion
of directional entropy and investigated the properties of the
complexities of these systems via directional entropies and their
entropy geometry.  Boyle and Lind pursued the study of the entropy geometry
further in \cite{BLind}.  Besides Milnor's examples we have many
examples whose directional entropies are finite and continuous in
all directions including irrational directions \cite{Milnor,Park1}. And they have the property
$\lim \limits_{n\rightarrow\infty}\frac{1}{n}{H}(\bigvee_{(i,j)\in
R_n}\sigma^{-(i,j)}{P})>0$. However as was shown in the above example (2), there are many $\Z^2-$actions
whose directional entropy does not capture the complexity of a
system.  We may say that the examples (2) and (3) have complexity in
the order of n, while positive entropy systems have complexity  in
the order of $n^2$.

Cassaigne constructed a uniformly recurrent point and hence a
minimal system of a given subexponential orbit growth rate \cite{C}.  In
\cite{DHP}, inspired by the viewpoint of ``topological independence"(see \cite{HY,KL}), the authors gave the definition of topological entropy
dimension to analyze the entropy zero systems. It measures the the
sub-exponential but sup-polynomial topological complexity via the
growth rate of orbits. Together with the examples, some of the properties of
the entropy zero topological systems have been investigated. As was
shown in physical models, many examples of low complexity do not
carry finite invariant measure. Their meaningful invariant measures
are $\sigma$-finite. Examples of finite measure preserving systems
of subexponential growth rate were first constructed in \cite{FP}.
Katok and Thouvenot introduced the notion of slow entropy for
$\Z^2-$actions to show that certain measure preserving $\Z^2-$actions
are not realizable by two commuting Lipschitz continuous maps in \cite{KT}.
Since the ``natural'' extension of the definition of entropy to slow entropy
is not an isomorphism invariant, they use the number of $\epsilon$-balls
in the Hamming distance to define the slow entropy. It is clear that their definition
is easily applied to $\mathbb{Z}$-actions to differentiate the complexity.

We will introduce the
notion of the entropy generating sequence and positive entropy
sequence to understand the complexity of entropy zero systems in section 2. By the definition,
it is clear that the entropy generating sequence
is a sequence along which the system has some independence. First we
define the dimension of a subset of $\N$ of density zero and use the notion to define
the entropy dimension of a system via the entropy generating sequence.
It is clear that the properties should be further investigated to
understand the structure of zero entropy systems.
We hope that many
of the tools developed for the study of positive entropy class are
to be investigated in the class of a given entropy dimension.
For example, we ask if we can have $\alpha$-dimension Pinsker $\sigma$-algebra and
$\alpha$-dimension Bernoulli in the case that $\alpha-$entropy exists. We ask also if we have
some kind of regularity in the size of the atoms of the iterated
partition of these systems. Moreover since general group actions
have many ``natural'' examples of entropy zero with diverse
complexity, we need to extend our study to general group actions. We
believe the study of entropy dimension together with the study of
subgroup actions will lead us to the understanding of more challenging and interesting properties of
entropy zero general group actions.

We briefly describe the content of the paper.
In section 2, we introduce the notion of entropy generating sequence
and positive entropy sequence which are subsets of $\N$. For a given
subset of $\N$, we introduce the notion of the dimensions, upper and
lower, of a subset to measure the ``size'' of the subset. This
notion classifies the ``size'' of the subsets of density 0. We
show(Proposition 2.4) the relation between the dimensions of entropy
generating sequence and positive entropy sequence. For a
measure-preserving system we will define the metric entropy
dimension through the dimensions of entropy generating sequence and
positive entropy sequence. We will study many of the basic
properties of entropy dimension. In section 3, we define the
dimension set of a system to understand the structure of the
complexity of its factors. We also introduce the notion of uniform
dimension whose dimension set consists of a singleton.  Using the dimension sets, we also study
the property of disjointness among entropy zero systems. We prove a
theorem which is more general than the disjointness between K-mixing
systems and zero entropy systems.  In section 4, for a compact metric
space we consider the entropy dimension of a given open cover with
respect to a measure and show that the topological entropy dimension
is always bigger than or equal to the metric entropy dimension of a
topological system. We provide a class of examples of
uniform dimension in section 5. In a rough statement, we may
say that the property without a factor of smaller
entropy dimension corresponds to the K-mixing property without
zero entropy factors.  Our construction is based on the cutting and stacking method as in \cite{FP},
but it demands technical arguments to guarantee that no partition has
smaller entropy dimension. We need to make level sets of each step ``spread out"
through the columns of the later towers without the increase of
the sub-exponential growth rate of orbits.

We mention that we noticed recently that the entropy dimension was
first introduced in \cite{Ca}. And another related concept ``scaled entropy"
was introduced to distinguish Bernoullian K-automorphisms with equal entropy by Vershik in \cite{V}.
For the study of completely integrable Hamiltonian systems, Marco [19] defined two entropy type invariants: polynomial entropy and weak polynomial entropy, which can 
be applied to measure polynomial scale of complexity. 
Since we have started our work on the complexity of topological and metric
entropy zero systems (\cite{ADP,DHP,DP}), there are several papers published in different directions
in the area (\cite{CL,H,M}).  Clearly this is the beginning of the study of entropy zero systems
with many more open questions.

\section{Entropy dimension}
Let $(X,\mathcal{B},\mu,T)$ be a measure-theoretical dynamical
system (MDS, for short) and $\alpha\in \mathcal{P}_X$, where
$\mathcal{P}_X$ denotes the collection of finite
measurable partitions of $X$.

In the case of zero entropy, we want to generalize the definition of entropy to measure the growth rate of the iterated partitions. However it has been noticed in
\cite{FP} that for $P\in \mathcal{P}_X$ the nature extension $C(T,P)=\inf \{\beta:\limsup_{n\rightarrow \infty}\frac{1}{n^{\beta}}H_{\mu}(\bigvee_{i=0}^{n-1}T^{-i}P)=0\}$ is not an isomorphic invariant.
More precisely, the following was proved. If there exists a partition $P$ such that
$C(T,P)=\inf \{\beta:\limsup_{n\rightarrow \infty}\frac{1}{n^{\beta}}H_{\mu}(\bigvee_{i=0}^{n-1}T^{-i}P)=0\}=\alpha>0$, then for any $\alpha<\tau<1$ and $\epsilon>0$,
there exists a partition $\tilde{P}$ such that
\begin{align*}
  (1).\ & |P-\tilde{P}|<\epsilon, \text{ and } \\
  (2).\ & \inf \{\beta:\limsup_{n\rightarrow \infty}\frac{1}{n^{\beta}}H_{\mu}(\bigvee_{i=0}^{n-1}\tilde{P})=0\}=\tau.
\end{align*}

To make $C(T,P)$ an isomorphic invariant, they count the number of $\epsilon-$balls in the Hamming distance of $n-$names
and take the limit of $n$'s and $\epsilon$'s \cite{FP}.

Before we introduce the notion of entropy dimension for a
measure-preserving system, we define the dimension of a subset $S$
of positive integers $\N$. Let $S=\{s_1<s_2<\cdots\}$ be an
increasing sequence of positive integers. For $\tau\ge 0$, we define
$$\overline{D}(S,\tau)=\limsup_{n\rightarrow
\infty}\frac{n}{(s_n)^\tau} \text{ and
}\underline{D}(S,\tau)=\liminf_{n\rightarrow
\infty}\frac{n}{(s_n)^\tau}.$$ It is clear that
$\overline{D}(S,\tau)\le \overline{D}(S,\tau')$ if $\tau\ge \tau'\ge
0$ and $\overline{D}(S,\tau) \notin \{0, +\infty\}$ for at most one
$\tau\ge 0$. We define {\it the upper dimension of $S$} by
\begin{align*}
\overline{D}(S)=\inf \{\tau\geq 0:\overline{D}(S,\tau)=0\}=\sup
\{\tau\geq 0:\overline{D}(S,\tau)=\infty \}.
\end{align*}
Similarly,  $\underline{D}(S,\tau)\le \underline{D}(S,\tau')$ if
$\tau\ge \tau'\ge 0$ and $\underline{D}(S,\tau) \notin \{0,
+\infty\}$ for at most one $\tau\ge 0$. We define {\it the lower
dimension of $S$} by
\begin{align*}
\underline{D}(S)=\inf \{\tau \geq 0:\underline{D}(S,\tau)=0\} =\sup
\{ \tau\geq 0:\underline{D}(S,\tau)=\infty \}.
\end{align*}
Clearly $0\le \underline{D}(S)\le \overline{D}(S)\le 1$. When
$\overline{D}(S)=\underline{D}(S)=\tau$, we say $S$ has dimension
$\tau$.  For example, if $S$ has positive density, then
$\overline{D}(S)=\underline{D}(S)=1$ and if
$S=\{n^2|n=1,2,\cdots\}$, then clearly
$\overline{D}(S)=\underline{D}(S)=\frac{1}{2}$.

In the following, we will investigate the dimension of a special
kind of sequences, which is called {\it the entropy generating
sequence}.

Let $(X,\mathcal{B},\mu,T)$ be a MDS and $\alpha\in \mathcal{P}_X$. We say an increasing sequence
$S=\{s_1<s_2<\cdots\}$ of $\N$ is an {\it entropy generating
sequence} of $\alpha$ if
$$\liminf_{n\rightarrow\infty}\frac{1}{n}H_\mu(\bigvee_{i=1}^nT^{-s_i}\alpha )>0.$$
We say $S=\{s_1<s_2<\cdots\}$ of $\N$ is a {\it positive entropy
sequence} of $\alpha$ if the {\it sequence entropy} of $\alpha$ along the sequence $S$, which is defined by
$$h^S_{\mu}(T,\alpha):=
\limsup_{n\rightarrow\infty}\frac{1}{n} H_\mu
(\bigvee_{i=1}^nT^{-s_i}\alpha),$$
is positive.

Denote by ${\cal E}_\mu(T,\alpha)$ the set of all entropy
generating sequences of $\alpha$, and $\mathcal{P}_\mu(T,\alpha)$
by the set of all positive entropy sequences of $\alpha$. Clearly
$\mathcal{P}_\mu(T,\alpha)\supset \mathcal{E}_\mu(T,\alpha)$.

\begin{definition} \label{de-ms-1} Let $(X,\mathcal{B},\mu,T)$ be a MDS and $\alpha\in
\mathcal{P}_X$. We define
$$\overline{D}^e_\mu(T,\alpha)=\begin{cases} \sup \limits_{S\in \mathcal{E}_\mu(T,\alpha)} \overline{D}(
S)\, &\text{if } \mathcal{E}_\mu(T,\alpha)\neq \emptyset\\ 0 \,
&\text{if } \mathcal{E}_\mu(T,\alpha)=\emptyset \end{cases},$$
$$\overline{D}^p_\mu(T,\alpha)=\begin{cases} \sup \limits_{S\in \mathcal{P}_\mu(T,\alpha)} \overline{D}(S)\,
&\text{if } \mathcal{P}_\mu(T,\alpha)\neq \emptyset\\ 0 \, &\text{if
} \mathcal{P}_\mu(T,\alpha)=\emptyset \end{cases}.$$ Similarly, we
define $\underline{D}^e_\mu(T,\alpha)$ and
$\underline{D}^p_\mu(T,\alpha)$ by changing the upper dimension into
lower dimension.
\end{definition}

\begin{definition}\label{de-ms-2} Let $(X,\mathcal{B},\mu,T)$ be a MDS.
We define
$$\overline{D}^e_\mu(X,T)=\sup_{\alpha \in \mathcal{P}_X}
\overline{D}^e_\mu(T,\alpha), \ \ \
\underline{D}^e_\mu(X,T)=\sup_{\alpha \in \mathcal{P}_X}
\underline{D}^e_\mu(T,\alpha),$$
$$\overline{D}^p_\mu(X,T)=\sup_{\alpha \in \mathcal{P}_X}
\overline{D}^p_\mu(T,\alpha), \ \ \
\underline{D}^p_\mu(X,T)=\sup_{\alpha \in \mathcal{P}_X}
\underline{D}^p_\mu(T,\alpha).$$
\end{definition}

Since the sequence entropies along a given sequence are the same for mutually conjugated systems, we can deduce that these four quantities are also conjugacy invariants. But the following
proposition shows that $\overline{D}^p_\mu(X,T)$ can only take trival values $0$ and $1$. A MDS $(X,\mathcal{B},\mu,T)$ is said to be {\it null} if
$h^S_{\mu}(T,\alpha)=0$ for any sequence $S$ of $\N$ and $\alpha \in
\mathcal{P}_X$. A well known result by Kushnirenko \cite{Ku} states that a MDS $(X,\mathcal{B},\mu,T)$ has discrete spectrum if and only if it is null.

\begin{proposition}\label{pro-dp-1} Let $(X,\mathcal{B},\mu,T)$ be a MDS. Then
$$\overline{D}^p_\mu(T,\alpha)=\begin{cases} 1\, &\text{if }
\mathcal{P}_\mu(T,\alpha)\neq \emptyset\\ 0 \, &\text{if }
\mathcal{P}_\mu(T,\alpha)=\emptyset \end{cases}   \ \ \text{ for }
\alpha\in \mathcal{P}_X.$$ Moreover, $\overline{D}^p_\mu(X,T)=0$ or
$1$, and $\overline{D}^p_\mu(X,T)=0$ if and only if
$(X,\mathcal{B},\mu,T)$ is null.
\end{proposition}
\begin{proof} When $\mathcal{P}_\mu(T,\alpha)=\emptyset$,
$\overline{D}^p_\mu(T,\alpha)=0$. Now assume
$\mathcal{P}_\mu(T,\alpha)\neq \emptyset$, thus there exists $S=\{
s_1<s_2<\cdots \}\subset \N$ such that
$$\limsup_{n\rightarrow +\infty} \frac{1}{n} H_\mu(\bigvee \limits_{i=1}^n
T^{-s_i}\alpha)=a>0.$$ Next we take $1\le n_1<n_2<n_3<\cdots$ such
that $n_{j+1}\ge 2 s_{n_j}$ for each $j\in \mathbb{N}$ and at the
same time $\limsup_{j\rightarrow +\infty} \frac{1}{n_j}
H_\mu(\bigvee \limits_{i=1}^{n_j} T^{-s_i}\alpha)=a$. Then put
$$F=S\cup \{1,2,\cdots,n_1\} \cup \bigcup_{i=1}^\infty \{
s_{n_i}+1,s_{n_i}+2,\cdots, n_{i+1}\}.$$ For simplicity, we write
$F=\{ f_1<f_2<\cdots\}$. Notice that
$$F\cap [1,s_{n_j}]\subset
[1,n_j]\cup(F\cap[n_j+1,s_{n_j}])\subset [1,n_j]\cup \{
s_1,s_2,\cdots,s_{n_j}\},$$ hence $|F\cap [1,s_{n_j}]|\le2n_j$. So
we have
\begin{align*}
&\hskip0.5cm \limsup_{n\rightarrow +\infty} \frac{1}{n}
H_\mu(\bigvee \limits_{i=1}^n T^{-f_i}\alpha) \ge
\limsup_{j\rightarrow +\infty} \frac{H_\mu(\bigvee
\limits_{i=1}^{n_j} T^{-s_i}\alpha)}{|F\cap
[1,s_{n_j}]|}\\
&\ge \limsup_{j\rightarrow +\infty} \frac{H_\mu(\bigvee
\limits_{i=1}^{n_j} T^{-s_i}\alpha)}{2n_j}=\frac{a}{2}>0,
\end{align*}
therefore $F\in \mathcal{P}_\mu(T,\alpha)$. Since $n_{j+1}\ge 2
s_{n_j}$ for each $j\in \mathbb{N}$, it is easy to see that
$\overline{D}(F)=1$. This implies $\overline{D}^p_\mu(T,\alpha)=1$
as $F\in \mathcal{P}_\mu(T,\alpha)$.
\end{proof}

In the following, we investigate the relations among these
dimensions.
\begin{proposition}\label{prop-kkk} Let $(X,\mathcal{B},\mu,T)$ be a MDS and $\alpha\in
\mathcal{P}_X$. Then
$$\underline{D}^e_\mu(T,\alpha)\le \overline{D}^e_\mu(T,\alpha)= \underline{D}^p_\mu(T,\alpha)
\le \overline{D}^p_\mu(T,\alpha).$$
\end{proposition}
\begin{proof} 1). $\underline{D}^e_\mu(T,\alpha)\le
\overline{D}^e_\mu(T,\alpha)$ and $\underline{D}^p_\mu(T,\alpha) \le
\overline{D}^p_\mu(T,\alpha)$ are obvious by Definition
\ref{de-ms-1}.

2). We will show that $\overline{D}^e_\mu(T,\alpha)\le
\underline{D}^p_\mu(T,\alpha)$. If $\overline{D}^e_\mu(T,\alpha)=0$,
then it is obvious that $\overline{D}^e_\mu(T,\alpha)\le
\underline{D}^p_\mu(T,\alpha)$. Now we assume that
$\overline{D}^e_\mu(T,\alpha)>0$, and $\tau\in
(0,\overline{D}^e_\mu(T,\alpha))$ is given.

There exists $S=\{ s_1<s_2<\cdots\}\in \mathcal{E}_\mu(T,\alpha)$
with $\overline{D}(S)>\tau$, i.e. $\limsup \limits_{n\rightarrow
+\infty} \frac{n}{s_n^\tau}=+\infty$. Hence
\begin{align}\label{eq-lim-1}
\limsup_{n\rightarrow +\infty} \frac{n}{n+s_n^\tau}=1.
\end{align}
Next we put $F=S\cup \{ \lfloor n^{\frac{1}{\tau}} \rfloor: n\in
\mathbb{N}\}$, where $\lfloor r\rfloor$ denotes the largest integer
less than or equal to $r$. Clearly $\underline{D}(F)\ge \tau$.

Let $F=\{ f_1<f_2<\cdots\}$. Then for each $n\in \mathbb{N}$ there
exists unique $m(n)\in \mathbb{N}$ such that $s_n=f_{m(n)}$. Since
$$\{ s_1,s_2,\cdots,s_n\}\subseteq \{
f_1,f_2,\cdots,f_{m(n)}\} \subseteq \{ s_1,s_2,\cdots,s_n \}\cup \{
\lfloor k^{\frac{1}{\tau}} \rfloor: k\le s_n^{\tau} \},$$ we have
$n\le m(n)\le n+s_n^\tau$. Combining this with \eqref{eq-lim-1}, we
get \begin{align}\label{eq-lim-2} \limsup_{n\rightarrow +\infty}
\frac{n}{m(n)}=1.
\end{align}

Now we have
\begin{align*}
\limsup_{m\rightarrow +\infty} \frac{H_\mu(\bigvee \limits_{i=1}^m
T^{-f_i}\alpha)}{m} &\ge \limsup_{n\rightarrow +\infty}
\frac{H_\mu(\bigvee \limits_{i=1}^{m(n)} T^{-f_i}\alpha)}{m(n)}\\
&\ge \limsup_{n\rightarrow +\infty} \frac{H_\mu(\bigvee
\limits_{i=1}^{n}
T^{-s_i}\alpha)}{n}\frac{n}{m(n)}\\
&\ge (\liminf_{n\rightarrow +\infty} \frac{H_\mu(\bigvee
\limits_{i=1}^{n} T^{-s_i}\alpha)}{n}) \cdot (\limsup_{n\rightarrow
+\infty}
\frac{n}{m(n)}) \\
&= \liminf_{n\rightarrow +\infty} \frac{H_\mu(\bigvee
\limits_{i=1}^{n}
T^{-s_i}\alpha)}{n}  \ \ \  \text{(by \eqref{eq-lim-2})}\\
&>0 \ \ \ \ \ \ \ \ \ (\text{since } S\in
\mathcal{E}_\mu(T,\alpha)).
\end{align*}
This implies $F\in \mathcal{P}_\mu(T,\alpha)$. Hence
$\underline{D}^p_\mu(T,\alpha)\ge \underline{D}(F)\ge \tau$. Since
$\tau$ is arbitrary in $(0,\overline{D}^e_\mu(T,\alpha))$, we have $\overline{D}^e_\mu(T,\alpha)\le
\underline{D}^p_\mu(T,\alpha)$.

\medskip
3). We need to prove that$\underline{D}^p_\mu(T,\alpha)\le
\overline{D}^e_\mu(T,\alpha)$.  If
$\underline{D}^p_\mu(T,\alpha)=0$, then it is obvious that
$\underline{D}^p_\mu(T,\alpha)\le \overline{D}^e_\mu(T,\alpha)$. Now
we assume that $\underline{D}^p_\mu(T,\alpha)>0$ and $\tau\in
(0,\underline{D}^p_\mu(T,\alpha))$ is given.

In the following, we show that

\noindent{\bf Fact A.} There exist a sequence $F=\{
f_1<f_2<\cdots\}$ of natural numbers and a real number $d>0$ such
that $\overline{D}(F)\ge \tau$ and for any $1\le m_1\le m_2$,
\begin{align}\label{eq-m1m2a}
H_\mu(\bigvee_{i=m_1}^{m_2}T^{-f_i}\alpha)\ge (m_2+1-m_1)d.
\end{align}
Moreover by \eqref{eq-m1m2a} we know $F\in
\mathcal{E}_\mu(T,\alpha)$. Hence $\overline{D}^e_\mu(T,\alpha)\ge
\overline{D}(F)\ge \tau$. Finally since $\tau$ is arbitrary, we have
$\overline{D}^e_\mu(T,\alpha)\ge \underline{D}^p_\mu(T,\alpha)$.

\medskip
Now it remains to prove Fact A. First, there exists $S=\{
s_1<s_2<\cdots\}\in \mathcal{P}_\mu(T,\alpha)$ with
$\underline{D}(S)>\tau$, i.e. $\liminf \limits_{n\rightarrow
+\infty} \frac{n}{s_n^\tau}=+\infty$. Hence there exists $a>0$ such
that
\begin{align}\label{ggj-eq0}
a n\ge s_n^\tau
\end{align}
 for all $n\in \mathbb{N}$.

Since $S\in \mathcal{P}_\mu(T,\alpha)$, there exist an increasing
sequence $\{n_1<n_2<\cdots<n_k<\cdots\}$ of positive integers and
$0<b<4$ such that $H_\mu(\bigvee \limits_{i=1}^{n_k}
T^{-s_i}\alpha)\ge n_k b$ for all $k\in \mathbb{N}$. Without loss of
generality(if necessary we choose a subsequence), we assume that
$n_{k+1}\ge \frac{4(H_\mu(\alpha)+1)}{b}\sum \limits_{j=1}^k n_j$
for all $k\in \mathbb{N}$. Let $c=\frac{b}{4(H_\mu(\alpha)+1)}$ and
$n_0=0$. Then $0<c<1$ and we have

\medskip
\noindent{\bf Claim:} For each $k\in \mathbb{N}$, there exist
$l_k\in \mathbb{N}$ and
$$F_k:=\{
i_1^k<i_2^k<\cdots<i_{l_k}^k\} \subseteq
\{n_{k-1}+1,n_{k-1}+2,\cdots,n_{k}\}$$ such that $cn_{k}\le l_k\le
n_{k}-n_{k-1}$ and $H_\mu(\bigvee \limits_{i\in F'_k} T^{-s_i}\alpha
)\ge |F_k'| \frac{b}{4}$ for each $\emptyset\neq F_k'\subseteq F_k$.

\begin{proof}[Proof of claim] Assume that the claim is not true.
Then for some $k\in \mathbb{N}$ there exist $w\in \mathbb{N}$ and $E_1,E_2,\cdots,E_w\subseteq
\{ n_{k-1}+1,n_{k-1}+2,\cdots,n_k\}$ such that $1\le
|E_1|,|E_2|,\cdots, |E_w|< cn_k$, $E_i\cap E_j=\emptyset$ for any
$1\le i< j\le w$ and $\bigcup \limits_{i=1}^w E_i=\{
n_{k-1}+1,n_{k-1}+2,\cdots,n_k\}$ and for $1\le j\le w-1$,
$H_\mu(\bigvee \limits_{t\in E_j} T^{-s_t}\alpha )<|E_j|
\frac{b}{4}$. This implies that
\begin{align*}
&\hskip0.5cm H_\mu(\bigvee_{i=1}^{n_k} T^{-s_i}\alpha)\le
H_\mu(\bigvee_{i=1}^{n_{k-1}} T^{-s_i}\alpha)+\sum_{j=1}^w
H_\mu(\bigvee_{t\in E_j} T^{-s_t}\alpha )\\
&\le n_{k-1} H_\mu(\alpha)+\sum_{j=1}^{w-1} |E_j|
\frac{b}{4}+|E_w|H_\mu(\alpha) \le \frac{b}{4}
n_{k}+\frac{b}{4}(n_k-n_{k-1})+cn_kH_\mu(\alpha)\\
&\le \frac{b}{4} n_{k}+\frac{b}{4}n_k+\frac{b}{4} n_k<bn_k,
\end{align*}
a contradiction.  This completes the proof of the Claim.
\end{proof}

Let $F=\bigcup \limits_{k=1}^\infty \{ s_i: i\in F_k\}$. For
simplicity, we write $F=\{ f_1<f_2<\cdots\}$. Then
\begin{align*}
\overline{D}(F,\tau)&=\limsup_{m\rightarrow +\infty}
\frac{m}{f_m^\tau}\ge \limsup_{v\rightarrow +\infty}
\frac{\sum_{k=1}^v l_k }{(f_{\sum_{k=1}^v
l_k})^\tau}=\limsup_{v\rightarrow +\infty} \frac{\sum_{k=1}^v l_k
}{(s_{i^v_{l_v}})^\tau}\\
&\ge \limsup_{v\rightarrow +\infty} \frac{l_v }{(s_{n_v})^\tau}\ge
\limsup_{v\rightarrow +\infty}
\frac{l_v }{an_v} \ \ \ \ (\text{by \eqref{ggj-eq0}})\\
&\ge \limsup_{v\rightarrow +\infty} \frac{c n_v }{an_v}\ge
\frac{c}{a}>0.
\end{align*}
Hence $\overline{D}(F)\ge \tau$.

For a given $m\in \mathbb{N}$, there exists a unique $k(m)\in
\mathbb{N}$ such that $\sum \limits_{k=0}^{k(m)-1} l_k< m \le\sum
\limits_{k=1}^{k(m)}l_k$, where $l_0=0$. Set $r(m)=m-\sum
\limits_{k=0}^{k(m)-1} l_k$. Then $f_m=s_{i_{r(m)}^{k(m)}}$. Now for
$1\le m_1\le m_2$, there are three cases.

\medskip

Case 1: $k(m_1)=k(m_2)$. Then
\begin{align*}
H_{\mu}(\bigvee_{i=m_1}^{m_2} T^{-f_i}\alpha ) &= H_\mu(
\bigvee_{j=r(m_1)}^{r(m_2)}  T^{-s_{i_j^{k(m_1)}}}\alpha) \ge \frac{b}{4} (r(m_2)+1-r(m_1))  \ \ \ (\text{by Claim})\\
&=\frac{b}{4} (m_2-m_1+1).
\end{align*}

\medskip

Case 2: $k(m_2)=k(m_1)+1$. Then
\begin{align*}
H_{\mu}(\bigvee_{i=m_1}^{m_2} T^{-f_i}\alpha ) &= H_\mu(
\bigvee_{j=r(m_1)}^{l_{k(m_1)}} T^{-s_{i_j^{k(m_1)}}}\alpha \vee
\bigvee_{j=1}^{r(m_2)} T^{-s_{i_j^{k(m_2)}}}\alpha)\\
&\ge \frac{1}{2}\left( H_\mu( \bigvee_{j=r(m_1)}^{l_{k(m_1)}}
T^{-s_{i_j^{k(m_1)}}}\alpha)+ H_\mu(\bigvee_{j=1}^{r(m_2)}
T^{-s_{i_j^{k(m_2)}}}\alpha) \right)\\
& \ge \frac{b}{8} \left( (l_{k(m_1)}+1-r(m_1))+r(m_2)
\right)=\frac{b}{8}(m_2-m_1+1).
\end{align*}

\medskip

Case 3: $k(m_2)\ge k(m_1)+2$. Then
\begin{align*}
H_{\mu}(\bigvee_{i=m_1}^{m_2} T^{-f_i}\alpha ) &\ge H_\mu(
\bigvee_{j=1}^{l_{k(m_2)-1}} T^{-s_{i_j^{k(m_2)-1}}}\alpha \vee
\bigvee_{j=1}^{r(m_2)} T^{-s_{i_j^{k(m_2)}}}\alpha) \\
&\ge \frac{1}{2}\left( H_\mu( \bigvee_{j=1}^{l_{k(m_2)-1}}
T^{-s_{i_j^{k(m_2)-1}}}\alpha)+ H_\mu(\bigvee_{j=1}^{r(m_2)}
T^{-s_{i_j^{k(m_2)}}}\alpha) \right)\\
& \ge \frac{b}{8} ( l_{k(m_2)-1}+r(m_2) )\ge
\frac{b}{8}(cn_{k(m_2)-1}+r(m_2)) \
\text{(by Claim)}\\
& \ge \frac{b}{8}(c\sum_{j=1}^{k(m_2)-1}l_j+cr(m_2))  \ \ \ \
\text{(by Claim)}\\
&=\frac{bc}{8}m_2\ge \frac{bc}{8}(m_2-m_1+1).
\end{align*}

Let $d=\frac{bc}{8}$. Then \eqref{eq-m1m2a} follows from the above
three cases.
\end{proof}

The following Theorem is a direct application of Proposition
\ref{prop-kkk}.
\begin{theorem}\label{thm-re} Let $(X,\mathcal{B},\mu,T)$ be a MDS, then
$\overline{D}_e(X,T)=\underline{D}_p(X,T)$.
\end{theorem}

By Proposition \ref{prop-kkk} and Theorem \ref{thm-re}, we have the
following definitions.
\begin{definition} \label{de-s-3} Let $(X,\mathcal{B},\mu, T)$ be a MDS and $\alpha\in
\mathcal{P}_X$. We define
$$\overline {D}_\mu(T,\alpha):=\overline{D}^e_\mu(T,\alpha)= \underline{D}^p_\mu(T,\alpha),$$
which is called {\it the upper entropy dimension of $\alpha$}. And we define $$\underline {D}_\mu(T,\alpha):=\underline{D}^e_\mu(T,\alpha)$$ to
be {\it the lower entropy dimension of $\alpha$}. When $\overline {D}_\mu(T,\alpha)=\underline {D}_\mu(T,\alpha)$, we note this quantity ${D}_\mu(T,\alpha)$,
{\it the entropy dimension of $\alpha$}.
\end{definition}

\begin{definition}\label{de-s-4} Let $(X,\mathcal{B},\mu, T)$ be a MDS.
We define
$$\overline {D}_\mu(X,T)=\sup_{\alpha \in \mathcal{P}_X}
\overline {D}_\mu(T,\alpha),$$ which is called {\it the upper metric entropy
dimension} of $(X,\mathcal{B},\mu,T)$. And we define $$\underline {D}_\mu(X,T)=\sup_{\alpha \in \mathcal{P}_X}
\underline {D}_\mu(T,\alpha),$$ which is called {\it the lower metric entropy
dimension} of $(X,\mathcal{B},\mu,T)$. When $\overline {D}_\mu(X,T)=\underline {D}_\mu(X,T)$, we denote the quantity by ${D}_\mu(X,T)$ and call it {\it the metric entropy
dimension} of $(X,\mathcal{B},\mu,T)$.
\end{definition}

By Proposition \ref{pro-dp-1}, we have
\begin{theorem} \label{null-zero} Let $(X,\mathcal{B},\mu,T)$ be a null MDS. Then
$D_\mu(X,T)=0$.
\end{theorem}

In the following, we study the basic properties of entropy dimension
of measure-preserving system. But since the upper dimension and the lower dimension do not agree in general,
we discuss the properties of the upper dimension. We note that they hold for the entropy dimension.
\begin{proposition}\label{prop-basic}
Let $(X,\mathcal{B},\mu, T)$ and $(Y,\mathcal{D},\nu, S)$ be two
MDS's and $\alpha,\beta\in \mathcal{P}_X, \eta\in \mathcal{P}_Y$.
Then
\begin{enumerate}
\item If $\alpha \preceq \beta$, then $\overline{D}_\mu(T,\alpha)\leq \overline{D}_\mu(T,
\beta)$, where by $\alpha \preceq \beta$ we mean that every atom of
$\beta$ is contained in one of the atoms of $\alpha$.

\item For any $0\le m \le n$, $\overline{D}_\mu(T,\alpha)=\overline{D}_\mu(T,\bigvee \limits_{i=m}^n T^{-i}\alpha)$.

\item  $\overline{D}_\mu(T,\alpha\vee \beta)=\max \{
\overline{D}_\mu(T,\alpha),\overline{D}_\mu(T,\beta)\}$.

\item $\overline{D}_\mu(X,T)=\sup \{ \overline{D}_\mu(T,\alpha):\alpha\in
\mathcal{P}_X^2\}$, where $\mathcal{P}_X^2$ denotes the set of
all partitions by two measurable sets of $X$.

\item $\overline{D}_{\mu\times \nu}(T\times S, \alpha\times \eta)
=\max \{ \overline{D}_\mu(T,\alpha), \overline{D}_\nu(S,\eta) \}$.
\end{enumerate}
Statements (1) and (2) also hold for lower dimensions.

\end{proposition}
\begin{proof} By the definition, (1) and (2) are obvious. Also (5) follows from (3).
For (3), firstly
we have $\overline{D}_\mu(T,\alpha\vee \beta)\ge \max \{
\overline{D}_\mu(T,\alpha),\overline{D}_\mu(T,\beta)\}$ by (1). Secondly, if
$\overline{D}_\mu(T,\alpha\vee \beta)=0$, then it is clear that
$\overline{D}_\mu(T,\alpha\vee \beta)=\max \{
\overline{D}_\mu(T,\alpha),\overline{D}_\mu(T,\beta)\}$. Now we assume that
$0<\overline{D}_\mu(T,\alpha\vee \beta)$. For any $\tau\in
(0,\overline{D}_\mu(T,\alpha\vee \beta))$. There exists $S=\{
s_1<s_2<\cdots\}\in \mathcal{P}_\mu(T,\alpha\vee \beta)$ with
$\underline{D}(S)>\tau$.

Since $S\in \mathcal{P}_\mu(T,\alpha\vee \beta)$, $\limsup
\limits_{n\rightarrow +\infty}\frac{1}{n}H_\mu(\bigvee
\limits_{i=1}^nT^{-s_i}(\alpha\vee \beta))>0$. This implies
$$\limsup \limits_{n\rightarrow +\infty}\frac{1}{n}H_\mu(\bigvee
\limits_{i=1}^nT^{-s_i}\alpha)>0 \text{ or }\limsup_{n\rightarrow
+\infty}\frac{1}{n}H_\mu(\bigvee_{i=1}^nT^{-s_i}\beta)>0,$$ that is,
$S\in \mathcal{P}_\mu(T,\alpha)$ or $S\in \mathcal{P}_\mu(T,\beta)$.
Hence $\tau\le \underline{D}(S)\le \max \{
\overline{D}_\mu(T,\alpha),\overline{D}_\mu(T,\beta)\} $. As $\tau$ is arbitary, we get
$\overline{D}_\mu(T,\alpha\vee \beta)=\max \{
\overline{D}_\mu(T,\alpha),\overline{D}_\mu(T,\beta)\}$.

Now we are to show (4). Clearly, $\overline{D}_\mu(X,T)\ge \sup \{
\overline{D}_\mu(T,\alpha):\alpha\in \mathcal{P}_X^2\}$. Conversely, for any
$\alpha=\{ A_1,\cdots,A_k\}\in \mathcal{P}_X$, let $\alpha_i=\{
A_i,A_i^c\}$ for $i=1,2,\cdots,k$. Then $\bigvee \limits_{i=1}^k
\alpha_i\succeq \alpha$. Hence by (1) and (3), we have
$$\overline{D}_\mu(T,\alpha)\le \max \{ \overline{D}_\mu(T,\alpha_i):1\le i\le k\}\le
\sup \{ \overline{D}_\mu(T,\alpha):\alpha\in \mathcal{P}_X^2\}.$$ Finally,
since $\alpha$ is arbitrary, we get (4).
\end{proof}
For two partitions $\alpha=\{ A_1,A_2,\cdots, A_k\}$ and  $\beta=\{ B_1,B_2,\cdots,B_k\}\in \mathcal{P}_X$,
denote by
$\mu(\beta\Delta \alpha):=\sum_{i=1}^k \mu(B_i\Delta A_i)$.

\begin{lemma} \label{lem-esti} Let $(X,\mathcal{B},\mu,T)$ be a MDS and
 $\alpha=\{ A_1,A_2,\cdots, A_k\}\in \mathcal{P}_X$.
Then for any $\epsilon>0$, there exists $\delta>0$ such that for any
$\beta=\{ B_1,B_2,\cdots,B_k\}\in \mathcal{P}_X$ with
$\mu(\beta\Delta \alpha)<\delta$,
it holds that
\begin{enumerate}
\item $\overline{D}_\mu(T,\beta)>\overline{D}_\mu(T,\alpha)-\epsilon$,
\item $\underline{D}_\mu(T,\beta)>\underline{D}_\mu(T,\alpha)-\epsilon$
and
\item
${D}_\mu(T,\beta)>{D}_\mu(T,\alpha)-\epsilon$ when the dimensions exist.
\end{enumerate}
\end{lemma}
\begin{proof} We only prove for upper dimension. If $\overline{D}_\mu(T,\alpha)=0$, it is obvious. Now assume
that $\overline{D}_\mu(T,\alpha)>0$. For any $\epsilon>0$, there exists $S=\{
s_1<s_2<s_3<\cdots\} \in \mathcal{\mathcal{E}}_\mu(T,\alpha)$ with
$\overline{D}(S)>\overline{D}_\mu(T,\alpha)-\epsilon$. There exists $\delta>0$ such that if
$\beta=\{ B_1,B_2,\cdots,B_k\}\in \mathcal{P}_X$ and
$\mu(\beta\Delta \alpha)<\delta$
then
$$H_\mu(\alpha|\beta)+H_\mu(\beta|\alpha)<\frac{1}{2}\liminf_{n\rightarrow\infty}\frac{1}{n}H_\mu(\bigvee_{i=1}^nT^{-s_i}\alpha )$$
(see Lemma 4.15 in \cite{Wal}).

For any $\beta=\{ B_1,B_2,\cdots,B_k\}\in \mathcal{P}_X$ and
$\mu(\beta\Delta \alpha)<\delta$,
\begin{align*}
&\liminf_{n\rightarrow +\infty}
\frac{1}{n}H_\mu(\bigvee_{i=1}^n T^{-s_i}\beta) \\
&\ge \liminf_{n\rightarrow +\infty} \frac{1}{n}\left(
H_\mu(\bigvee_{i=1}^n T^{-s_i}(\alpha\vee \beta))
-H_\mu(\bigvee_{i=1}^n T^{-s_i}\beta|\bigvee_{i=1}^n T^{-s_i}\alpha)\right)\\
&\ge \liminf_{n\rightarrow +\infty} \frac{1}{n}\left(
H_\mu(\bigvee_{i=1}^n T^{-s_i}\alpha )-nH_\mu(\beta|\alpha) \right
)\\
&\ge \frac{1}{2}\liminf_{n\rightarrow\infty}\frac{1}{n}H_\mu(\bigvee_{i=1}^nT^{-s_i}\alpha )>0,
\end{align*}
that is, $S\in \mathcal{E}_\mu(T,\beta)$. Hence $\overline{D}_\mu(T,\beta)\ge
\overline{D}(S)>\overline{D}_\mu(T,\alpha)-\epsilon$.
\end{proof}

\begin{theorem}\label{thm-mart} Let $(X,\mathcal{B},\mu,T)$ be a MDS.
\begin{enumerate}
\item If $\{ \alpha_i\}_{i\in \mathbb{N}}\subset \mathcal{P}_X$ and $\alpha\in \mathcal{P}_X$ satisfying
 $\alpha\preceq \bigvee \limits_{i\in \mathbb{N}}\alpha_i$, then
$\overline{D}_\mu(T,\alpha)\le \sup \limits_{i\ge 1} \overline{D}_\mu(T,\alpha_i)$.

\item If $\{ \alpha_i\}_{i\in \mathbb{N}}\subset \mathcal{P}_X$ and
$\alpha_i\nearrow \mathcal{B}\ (\text{\rm mod}\, \mu)$, then $\overline{D}_\mu(X,T)=\lim
\limits_{i\rightarrow +\infty} \overline{D}_\mu(T,\alpha_i)$. Moreover, if
$\alpha$ is a generating partition, i.e. $\bigvee
\limits_{i=0}^\infty T^{-i}\alpha=\mathcal B\ (\text{\rm mod}\, \mu)$, then
$\overline{D}_\mu(X,T)=\overline{D}_\mu(T,\alpha)$.

\item If $\{ \alpha_i\}_{i\in \mathbb{N}}\subset \mathcal{P}_X$ and
$\bigvee\limits_{i\in\N}\alpha_i=\mathcal B\ (\text{\rm mod}\, \mu)$, then
$\overline{D}_\mu(X,T)=\sup\limits_{i\ge 1}\overline{D}_\mu(T,\alpha_i)$.
\end{enumerate}
\end{theorem}
\begin{proof} It is obvious that (1) implies (2) and (3). Now we are to show (1).
Let $\alpha=\{ A_1,A_2,\cdots,$ $A_k\}\in \mathcal{P}_X$ and fix
$\epsilon>0$. By Lemma \ref{lem-esti}  there exists $\delta>0$ such
that if $\beta=\{ B_1,B_2,\cdots,B_k\}\in \mathcal{P}_X$ and
$\mu(\beta\Delta \alpha):=\sum \limits_{i=1}^k \mu(B_i\Delta
A_i)<\delta$, then $\overline{D}_\mu(T,\beta)>\overline{D}_\mu(T,\alpha)-\epsilon$. Since
$\alpha\preceq \bigvee \limits_{i\in \mathbb{N}}\alpha_i$, there
exist $N\in \mathbb{N}$ and $\gamma=\{ C_1,C_2,\cdots,C_k\}\preceq
\bigvee \limits_{i=1}^N \alpha_i$ such that $\mu(\gamma\Delta
\alpha)<\delta$. Thus $\overline{D}_\mu(T,\gamma)\ge \overline{D}_\mu(T,\alpha)-\epsilon$
and so
\begin{align*}
\sup_{i\ge 1} \overline{D}_\mu(T,\alpha_i)&\ge \max \{
\overline{D}_\mu(T,\alpha_i)|i=1,2,\cdots,N\}=\overline{D}_\mu(T,\bigvee_{i=1}^N
\alpha_i)\\ &\ge \overline{D}_\mu(T,\gamma)\ge \overline{D}_\mu(T,\alpha)-\epsilon.
\end{align*}
Since the above inequality is true for any
$\epsilon>0$, we get $\sup\limits_{i\ge 1} \overline{D}_\mu(T,\alpha_i)\ge
\overline{D}_\mu(T,\alpha)$.
\end{proof}

\begin{proposition}\label{pro-D-ite} Let $(X,\mathcal{B},\mu,T)$ be a MDS and $\alpha\in \mathcal{P}_X$.
\begin{enumerate}
\item For $k\in \mathbb{N}$, we have
$\overline{D}_\mu(T^k,\alpha)=\overline{D}_\mu(T,\alpha)$. Moreover,
$\overline{D}_\mu(X,T^k)=\overline{D}_\mu(X,T)$ and this is also true for lower dimensions and dimensions whenever the dimensions exist.

\item When $T$ is invertible, we have $\overline{D}_\mu(T,\alpha)=\overline{D}_\mu(T^{-1},\alpha)$ and hence $\overline{D}_\mu(X,T)=\overline{D}_\mu(X,T^{-1})$.
\end{enumerate}
\end{proposition}
\begin{proof} (1). We only prove for the upper dimension. Given $k\in \mathbb{N}$. Let $S\in \mathcal{E}_\mu(T^k,\alpha)$. Then
$kS=\{ ks:s\in S\}\in \mathcal{E}_\mu(T,\alpha)$. Since
$\overline{D}(kS)=\overline{D}(S)$, $\overline{D}_\mu(T,\alpha)\ge
\underline{D}(S)$. Finally since $S$ is arbitrary in $\mathcal{E}_\mu(T,\alpha)$,
$\overline{D}_\mu(T,\alpha)\ge \overline{D}_\mu(T^k,\alpha)$.

Conversely, let $S=\{ s_1<s_2<\cdots\}\in
\mathcal{E}_\mu(T,\alpha)$. Without loss of generality, we assume
$s_1\ge k$. Set $S_1=\{ \lfloor\frac{s_i}{k}\rfloor:i\in \mathbb{N}
\}$. For simplicity of the notation, we write $S_1=\{ t_1<t_2<\cdots\}$. Then $\lfloor
\frac{s_j}{k}\rfloor \le t_j\le \lfloor \frac{s_{j}+k}{k}\rfloor$ for
all $j\in \mathbb{N}$.

Now
\begin{align*}
\liminf_{n\rightarrow +\infty} \frac{H_\mu(\bigvee \limits_{j=1}^n
T^{-kt_j}\alpha)}{n}&\ge \liminf_{n \rightarrow +\infty}
\frac{H_\mu(\bigvee \limits_{i=0}^{k-1}\bigvee \limits_{j=1}^{n}
T^{-(kt_j+i)}\alpha)}{kn}\ge \liminf_{n \rightarrow +\infty}
\frac{H_\mu(\bigvee \limits_{j=1}^{n} T^{-s_j}\alpha)}{kn}>0,
\end{align*}
as $S\in \mathcal{E}_\mu(T,\alpha)$. This implies that $S_1\in
\mathcal{E}_\mu (T^k,\alpha)$. Since
$\overline{D}(S)=\overline{D}(S_1)$, $\overline{D}_\mu(T^k,\alpha)\ge
\overline{D}(S_1)=\overline{D}(S)$. Finally since $S$ is
arbitrary, $\overline{D}_\mu(T^k,\alpha)\ge \overline{D}_\mu(T,\alpha)$.

\medskip
(2). Let $T$ be invertible. By symmetry of $T$ and $T^{-1}$, it is
sufficient to show $\overline{D}_\mu(T^{-1},\alpha)\ge \overline{D}_\mu(T,\alpha)$. If
$\overline{D}_\mu(T,\alpha)=0$, this is obvious. Now we assume that
$\overline{D}_\mu(T,\alpha)>0$. Given $\tau\in (0,\overline{D}_\mu(T,\alpha))$.

By Fact A in the proof of Proposition \ref{prop-kkk}, we know that
there exists a sequence $S=\{ s_1<s_2<\cdots\}\subset\N$ and $a>0$
such that $\overline{D}(S)>\tau$ and for any $1\le m_1\le m_2$,
\begin{align}\label{eq-m1m2}
H_\mu(\bigvee_{i=m_1}^{m_2}T^{-s_i}\alpha)\ge (m_2+1-m_1)a.
\end{align}
Since $\overline{D}(S)>\tau$, there exists a sequence $\{
n_1<n_2<\cdots\}$ such that $n_1\ge 2$, $n_{i+1}\ge
1+2\sum_{j=1}^in_i$ and $n_i\ge s_{n_i}^\tau$ for all $i\in
\mathbb{N}$. Let $n_0=0$, $s_0=0$, $f_0=0$ and
$f_m=s_{n_{j}}-s_{n_{j}-m}$ if $n_{j-1}< m\le n_{j}$ for some $j\in
\mathbb{N}$. Put $F=\{ f_m:m\in \mathbb{N}\}$.

Set $n_{-1}=0$. Given $n\in \mathbb{N}$ with $n\ge n_1+1$, there
exists $j\ge 2$ such that $n_{j-1}< n\le n_{j}$. Now we have
\begin{align*}
H_\mu(\bigvee_{m=1}^n T^{f_m}\alpha)&\ge \max \{
H_\mu(\bigvee_{m=n_{j-2}+1}^{n_{j-1}} T^{f_m}\alpha ),
H_\mu(\bigvee_{m=n_{j-1}+1}^n T^{f_m}\alpha )\}\\
&=\max \{
H_\mu(\bigvee_{m=n_{j-2}+1}^{n_{j-1}}T^{s_{n_{j-1}}-s_{n_{j-1}-m}}\alpha),
H_\mu(\bigvee_{m=n_{j-1}+1}^{n}T^{s_{n_j}-s_{n_j-m}}\alpha )\}\\
&=\max \{
H_\mu(\bigvee_{m=n_{j-2}+1}^{n_{j-1}}T^{-s_{n_{j-1}-m}}\alpha),
H_\mu(\bigvee_{m=n_{j-1}+1}^{n}T^{-s_{n_j-m}}\alpha )\}\\
&\ge \max\{
a(n_{j-1}-n_{j-2}-1),a(n-n_{j-1})\}  \ \ \ \ \  \text{(by \eqref{eq-m1m2})} \\
&\ge \max\{ \frac{a}{2}n_{j-1}, \frac{a}{2}(n-n_{j-1})\}\\
&\ge\frac{a}{4} (n_{j-1}+n-n_{j-1})\\
&=\frac{a}{4}n,
\end{align*}
that is,
\begin{align}\label{dfx-eq-1}
H_\mu(\bigvee_{m=1}^n T^{f_m}\alpha)\ge \frac{a}{4}{n}.
\end{align}
Since \eqref{dfx-eq-1} is true for any $n\ge n_1+1$, $F\in
\mathcal{E}_\mu(T^{-1},\alpha)$. Now note that $f_{n_j}=s_{n_j}$ for
$j\in \mathbb{N}$, one has
\begin{align*}
\overline{D}(F,\tau)&=\limsup_{m\rightarrow +\infty}
\frac{m}{f_m^\tau}\ge \limsup_{j\rightarrow +\infty} \frac{n_j
}{s_{n_j}^\tau}\ge 1.
\end{align*}
Hence $\overline{D}(F)\ge \tau$. Moreover $\overline{D}_\mu(T^{-1},\alpha)\ge
\overline{D}(F)\ge \tau$ as $F\in \mathcal{E}_\mu(T^{-1},\alpha)$.
Finally since $\tau$ is arbitrary in $(0,\overline{D}_\mu(T,\alpha))$, we have $\overline{D}_\mu(T^{-1},\alpha)\ge
\overline{D}_\mu(T,\alpha)$.
\end{proof}

\section{factors and joinings}

In this section, we will introduce several notions like dimension
sets, dimension $\sigma$-algebras and uniform dimension systems to
understand the structure of entropy zero systems.

When the metric entropy dimension of a MDS exists, the entropy dimensions of its factors still may not exists. One of the easy examples is a product
system, one with entropy dimension but the other system with no entropy dimension. So in this section we are to consider the upper entropy dimension.

\begin{definition}\label{definition-dset}
We define the {\it dimension set} of a MDS $(X,\mathcal{B},\mu,T)$
by
\begin{align*}
  Dims_\mu(X,T)&=\{\overline{D}_\mu(T,\{ A,X\setminus A\}): A\in \mathcal{B} \text{ and } 0<\mu(A)<1\}\\
  &=\{\overline{D}_\nu(Y,S): (Y,\mathcal{D},\nu,S) \text{ is a factor of } (X,\mathcal{B},\mu,T)\}.
\end{align*}
\end{definition}

\begin{remark}
It is clear that $Dims_\mu(X,T)=\emptyset$ if and only if
$(X,\mathcal{B},\mu, T)$ is a trivial system, i.e., $\mathcal{B}=\{
\emptyset,X\} (\bmod \mu)$. We use the convention $\sup\{\tau\in
Dims_\mu(X,T)\}=0$ when $Dims_\mu(X,T)=\emptyset$. Thus
$\overline{D}_\mu(X,T)=\sup\{\tau\in Dims_\mu(X,T)\}$.
\end{remark}

Let $(X,\mathcal{B},\mu,T)$ be a MDS. Let $$P_\mu(T)=\{ A\in \mathcal{B}: \text{ the measure-theoretic entropy of } \{ A,X\setminus A\} \text{ is zero} \}.$$
It is a $T$-invariant sub $\sigma$-algebra of $\mathcal{B}$ and is known as the {\it Pinsker $\sigma$-algebra} of
$(X,\mathcal{B},\mu,T)$. Moreover, for $k\ge 1$, $P_{\mu}(T)=P_{\mu}(T^k)$. If in addition $T$ is invertible,
then $P_\mu(T)=P_\mu(T^{-1})$.
Recall that the MDS $(X,\mathcal{B},\mu,T)$ is said to have {\it completely positive entropy} (c.p.e., for short) if $P_{\mu}(T)=\{X,\emptyset\}$.
By the well-known Rohlin-Sinai's Theorem (\cite{RS}), a MDS is a K-system if and only if it has c.p.e.. For more details, one may see Parry or Glasner's books (\cite{Parry,Glasner})
for references.

For $\tau\in [0,1)$, we
define
$$P^\tau_\mu(T):=\{ A\in \mathcal{B}: \overline{D}_\mu(T,\{ A,X\setminus A\})\le \tau \}.$$
It is clear that $P^{\tau_1}_\mu(T)\subseteq
P^{\tau_2}_\mu(T)\subseteq P_\mu(T)$ for any $0\le \tau_1\le
\tau_2<1$. The following theorem lists some basic properties of $P^\tau_\mu(T)$.

\begin{theorem}
 \label{h-thm-pinsker-1} Let $(X,\mathcal{B},\mu,T)$ be a MDS and $\tau\in [0,1)$. Then
\begin{enumerate}
\item $P^\tau_{\mu}(T)$ is a sub $\sigma$-algebra of
$\mathcal{B}$.
\item $T^{-1}P^\tau_{\mu}(T)=P^\tau_\mu(T)(mod \, \mu)$.
\item For $k\ge 1$, $P^\tau_{\mu}(T)=P^\tau _{\mu}(T^k)$. If $T$ is invertible,
then $P^\tau_\mu(T)=P^\tau_\mu(T^{-1})$.
\end{enumerate}
\end{theorem}
\begin{proof} (1). Clearly, $\emptyset, X\in P_\mu^\tau (T)$.
Let $A,B\in P^\tau_\mu(T)$. Since $X\setminus (X\setminus A)=A$,
$X\setminus A\in P^\tau_\mu(T)$.

Let $A_i\in P^\tau_\mu(T)$, $i\in \mathbb{N}$. Now we are to show
that $\cup_{i=1}^\infty A_i\in P^\tau_\mu(T)$, i.e., $$\overline{D}_\mu(T,\{
\cup_{i=1}^\infty A_i, X\setminus \cup_{i=1}^\infty A_i \})\le
\tau.$$ Since $\{ \cup_{i=1}^\infty A_i, X\setminus
\cup_{i=1}^\infty A_i \}\subseteq \bigvee_{i=1}^\infty \{
A_i,X\setminus A_i\}$, using Theorem \ref{thm-mart} (1) we get
$$\overline{D}_\mu(T,\{ \cup_{i=1}^\infty A_i, X\setminus \cup_{i=1}^\infty A_i \})\le \sup_{i\in
\mathbb{N}} \overline{D}_\mu(T, \{ A_i,X\setminus A_i\})\le \tau.$$ Hence
$\cup_{i=1}^\infty A_i \in P^\tau_\mu(T)$. This shows
$P^\tau_{\mu}(T)$ is a sub $\sigma$-algebra of $\mathcal{B}$.

\medskip
(2). Since for $A\in \mathcal{B}$, $$\overline{D}_\mu(T,\{ A,X\setminus
A\})=\overline{D}_\mu(T,T^{-1}\{ A,X\setminus A\})= \overline{D}_\mu(T,\{
T^{-1}A,X\setminus T^{-1}(A)\}),$$ we have
$T^{-1}P^\tau_\mu(T)\subseteq P^\tau_\mu(T)$.

Conversely, let $A\in P^\tau_\mu(T)$. Then $A\in
P^\tau_\mu(T)\subseteq P_\mu(T)=T^{-1}P_\mu(T)$. Hence there
exists $B\in \mathcal{B}$ such that $A=T^{-1}B$. Now note that
$$\overline{D}_\mu(T,\{ B,X\setminus B\})=\overline{D}_\mu(T,\{ T^{-1}B, T^{-1}(X\setminus B)\})=\overline{D}_\mu(T,\{
A,X\setminus A\})\le \tau,$$ we have $B\in P^\tau_\mu(T)$, i.e.,
$A\in T^{-1}P^\tau_\mu(T)$. Therefore
$P^\tau_\mu(T)=T^{-1}P^\tau_\mu(T)$.

(3). For $\alpha\in \mathcal{P}_X$ and $k\in \mathbb{N}$, we have
$\overline{D}_\mu(T^k,\alpha)=\overline{D}_\mu(T,\alpha)$(see Proposition
\ref{pro-D-ite}). Hence $\alpha\subset P^\tau_\mu(T)$ if and only if
$\alpha\subset P^\tau_\mu(T^k)$. This implies
$P^\tau_\mu(T)=P^\tau_\mu(T^k)$. Finally,
$P^\tau_\mu(T)=P^\tau_\mu(T^{-1})$ by Proposition \ref{pro-D-ite}.
\end{proof}

We call $P^\tau_\mu(T)$ the {\it $\tau^-$-dimension sub
$\sigma$-algebra} of $(X,\mathcal{B},\mu,T)$, and if
$P^\tau_\mu(T)=\mathcal{B}$, we call $(X,\mathcal{B},\mu,T)$ a {\it
$\tau^-$-dimension system}.

The following theorem states that the entropy dimension set must be
right closed.
\begin{theorem}
Let $(X,\mathcal{B},\mu,T)$ be an invertible ergodic MDS. If $r_i\in
Dims_\mu(X,T)$, $i\in \mathbb{N}$ and $r\in [0,1]$ such that
$r_i\nearrow r$, then $r\in Dims_\mu(X,T)$.
\end{theorem}
\begin{proof} For $i\in \mathbb{N}$, let $P_i=\{A_i,X\setminus A_i\}$ for
some $A_i\in \mathcal{B}$ with $0<\mu(A_i)<1$ such that $r_i=\overline{D}_\mu
(T,P_i)\in Dims_\mu(X,T)$. We denote by $\mathcal {B} _i=\bigvee
\limits_{n=-\infty}^\infty T^{-n}P_i$, the $\sigma-$algebra
generated by $P_i$. Let ${\mathcal D}=\bigvee
\limits_{i=1}^\infty\mathcal {B}_i$. For each $i\in \mathbb{N}$,
$h_\mu(T,P_i)=0$ since $r_i<1$. Thus $h_{\mu}(T,\mathcal D)=0$,
moreover by Krieger's generator theorem\cite{K}, we have a partition
$P=\{A,X\setminus A\}$ such that $\mathcal D=\bigvee_{n=-\infty}^\infty
T^{-n}P$. Then by Theorem \ref{thm-mart} (1) and Proposition 2.9 (2) we have $$r_i\le\sup \limits_{n\ge 1} \overline{D}_\mu(T,\bigvee_{k=-n}^n
T^{-k}P)=\overline{D}_{\mu}(T,P),$$ for all $i=1,2,\cdots$. Since
$P \preceq \bigvee \limits_{n\in \mathbb{N}}\bigvee \limits_{j=-n}^n\bigvee \limits_{i=1}^nT^{-j}P_i$,
by Theorem \ref{thm-mart} (1) again, $$\overline{D}_{\mu}(T,P)\le \sup \limits_{n\ge 1} \overline{D}_\mu(T,\bigvee \limits_{j=-n}^n\bigvee \limits_{i=1}^nT^{-j}P_i)=r.$$
Hence $\overline{D}_\mu(T,P)=r$ which shows that $r\in Dims_\mu(X,T)$.
\end{proof}

In the following we will give a disjointness theorem via entropy
dimension. Let's recall the related notions first.

Let $(X,\mathcal{B},\mu,T)$ and $(Y,\mathcal{D},\nu,S)$ be two
MDS's. A probability measure $\lambda$ on $(X\times
Y,\mathcal{B}\times \mathcal{D})$ is a {\it joining} of
$(X,\mathcal{B},\mu,T)$ and $(Y,\mathcal{D},\nu,S)$ if it is
$T\times S$-invariant, and has $\mu$ and $\nu$ as marginals; i.e.
$proj_X(\lambda)=\mu$ and $proj_Y(\lambda)=\nu$. We let $J(\mu,\nu)$
be the space of all joinings of $(X,\mathcal{B},\mu,T)$ and
$(Y,\mathcal{D},\nu,S)$. We say $(X,\mathcal{B},\mu,T)$ and
$(Y,\mathcal{D},\nu,S)$ are {\it disjoint} if
$J(\mu,\nu)=\{\mu\times \nu\}$. More generally if $\{
(X_i,\mathcal{B}_i,\mu_i,T_i)\}_{i\in I}$ is a collection of MDS's,
a probability measure $\lambda$ on $(\prod_{i\in I} X_i,\prod_{i\in
I}\mathcal{B}_i)$ is a {\it joining} of
$\{(X_i,\mathcal{B}_i,\mu_i,T_i)\}$ if it is $\prod_{i\in
I}T_i$-invariant, and has $\mu_i$ as marginals; i.e.
$proj_{X_i}(\lambda)=\mu_i$ for every $i\in I$. We let $J(\{\mu_i
\}_{i\in I})$ be the spaces of all these joinings.  When
$(X_i,\mathcal{B}_i,\mu_i,T_i)=(X,\mathcal{B},\mu,T)$ for $i\in I$,
we write $J(\{\mu_i \}_{i\in I})$ as $J(\mu;I)$ and call $\lambda\in
J(\mu;I)$ $I$-fold self-joinings.

\begin{lemma} \label{lem-join-dim}
Let $(X,\mathcal{B}, \mu,T)$ be a MDS. If $\eta\in J(\mu;\mathbb{Z})$, then
$\overline{D}_\eta(X^{\mathbb{Z}},T^{\mathbb{Z}})=\overline{D}_\mu(X,T)$.
\end{lemma}
\begin{proof}  First there exists $\{\alpha_j\}_{j=1}^\infty \subseteq \mathcal{P}_X$
such that $\alpha_1\preceq \alpha_2\preceq \alpha_3\cdots$ and
$\bigvee_{j=1}^\infty \alpha_j=\mathcal{B} (\mod \mu)$. Then for
$i\in \mathbb{Z}$ let $\pi_i: X^{\mathbb{Z}}\rightarrow X$ be the
$i$-th coordinate projection. Let $\beta_i^j=\pi_i^{-1}(\alpha_j)$
for $i\in \mathbb{Z}$ and $j\in \mathbb{N}$. Then $\beta_i^j\in
\mathcal{P}_{X^{\mathbb{Z}}}$. It is clear that $\bigvee_{i\in
\mathbb{Z},j\in \mathbb{N}} \beta^j_i=\mathcal{B}^{\mathbb{Z}} (\mod
\eta)$ and $\overline{D}_\mu(T,\alpha_j)=\overline{D}_\eta(T^{\mathbb{Z}},\beta_i^j)$ for $i\in
\mathbb{Z},j\in \mathbb{N}$. Hence by Theorem \ref{thm-mart} (1),
$$\overline{D}_\eta(X^{\mathbb{Z}},T^{\mathbb{Z}})=\sup_{i\in
\mathbb{Z},j\in \mathbb{N}}
\overline{D}_\eta(T^{\mathbb{Z}}, \beta_i^j )=\sup_{j\in
\mathbb{N}}\overline{D}_\mu(T,\alpha_j)=\overline{D}_\mu(X,T).$$ This finishes the proof
of the Lemma.
\end{proof}

\begin{theorem}\label{dis}
Let $(X,\mathcal{B},\mu,T)$ be an invertible MDS and
$(Y,\mathcal{D},\nu,S)$ be an ergodic MDS. If
$Dims_\mu(X,T)>\overline{D}_\nu(Y,S)$ (i.e. for any $\tau\in
\text{Dims}_\nu(X,T)$, $\tau>\overline{D}_\nu(Y,S)$), then
$(X,\mathcal{B},\mu,T)$ is disjoint from $(Y,\mathcal{D},\nu,S)$.
\end{theorem}
\begin{proof} We follow the arguments in the proof of Theorem 1 in \cite{GTW}.
Let $\lambda$ be a joining of $(X,\mathcal{B},\mu,T)$ and
$(Y,\mathcal{D},\nu,S)$. Let
$$\lambda=\int_X \delta_x\times \lambda_x d \mu(x)$$
be the disintegration of $\lambda$ over $\mu$ and define probability
measure $\lambda_\infty$ on $X\times Y^{\mathbb{Z}}$ and
$\nu_\infty$ on $Y^{\mathbb{Z}}$ by:
$$\lambda_\infty=\int_X \delta_x \times (\cdots\times\lambda_x\times \lambda_x\cdots ) d \mu(x)$$ and
$$\nu_\infty=\int_X (\cdots\times\lambda_x\times \lambda_x\cdots ) d
\mu(x).$$ Since $\lambda$ is $T\times S$-invariant,
\begin{align*}
\lambda&=(T\times S)\lambda=\int_X \delta_{Tx}\times S\lambda_x d
\mu(x)\\
&=\int_X \delta_{x}\times S\lambda_{T^{-1} x} d \mu(x).
\end{align*}
By uniqueness of disintegration we have
$\lambda_x=S\lambda_{T^{-1}x}$ for $\mu$-a.e. $x\in X$, i.e.,
$S\lambda_x=\lambda_{Tx}$ for $\mu$-a.e. $x\in X$. Moreover
\begin{align*}
S^{\mathbb{Z}} \nu_\infty=\int_X (\cdots\times S\lambda_x\times
S\lambda_x\cdots ) d \mu(x)= \int_X (\cdots\times \lambda_{Tx}\times
\lambda_{Tx}\cdots ) d \mu(x)=\nu_\infty
\end{align*}
This implies $\nu_\infty\in J(\nu,\mathbb{Z})$ since $\int_X
\lambda_x d\mu(x)=\nu$. It is also clear that $\lambda_\infty \in
J(\{ \mu, \nu_\infty\})$, i.e., $\lambda_\infty$ is a joining of
$(X,\mathcal{B},\mu,T)$ and $(Y^{\mathbb{Z}},
\mathcal{D}^{\mathbb{Z}},\nu_\infty,S^{\mathbb{Z}})$.

Let $$\mathcal{E}=\{ E\in \mathcal{B}: \exists F\in
\mathcal{D}^{\mathbb{Z}},\, \lambda_\infty ((E\times
Y^{\mathbb{Z}})\Delta (X\times F))=0\}$$ and
$$\mathcal{F}=\{F\in \mathcal{D}^{\mathbb{Z}}: \exists E\in
\mathcal{B},\, \lambda_\infty ((E\times Y^{\mathbb{Z}})\Delta
(X\times F))=0\}.$$ Then $\mathcal{E}$ is a $T$-invariant
sub-$\sigma$-algebra and $\mathcal{F}$ is a
$S^{\mathbb{Z}}$-invariant sub-$\sigma$-algebra.

Now for any $E\in \mathcal{E}$ there exists $F\in
\mathcal{D}^{\mathbb{Z}}$ such that $\lambda_\infty ((E\times
Y^{\mathbb{Z}})\Delta (X\times F))=0$. Now
\begin{align*}
\overline{D}_\mu(T,\{ E,X\setminus E\})&=\overline{D}_{\lambda_\infty}(T\times
S^{\mathbb{Z}}
, \{ E\times Y^{\mathbb{Z}},(X\setminus E)\times Y^{\mathbb{Z}} \}\\
&=\overline{D}_{\lambda_\infty}(T\times S^{\mathbb{Z}},\{ X\times F,X\times
(Y^{\mathbb{Z}}\setminus F)\}\\
& =\overline{D}_{\nu_\infty} (S^{\mathbb{Z}},\{ F, Y^{\mathbb{Z}}\setminus
F)\}\\& \le \overline{D}_{\nu_\infty}(Y^{\mathbb{Z}},S^{\mathbb{Z}})\\
&=\overline{D}_{\nu}(Y,S)    \ \ \ \ \ (\text{by Lemma \ref{lem-join-dim}}).
\end{align*}
Since $Dims_\mu(X,T)>\overline{D}_\nu(Y,S)$, we have $\mu(E)=0$ or $1$. Hence
$\mathcal{E}=\{ \emptyset, X\} \, (\bmod\, \mu)$ and so
$\mathcal{F}=\{ \emptyset, Y^{\mathbb{Z}}\} \, (\bmod \,
\nu_\infty)$.

Define a transformation $R: X\times Y^{\mathbb{Z}}\rightarrow
X\times Y^{\mathbb{Z}}$ by $R(x,{\bf y})=(x,\sigma { \bf y}) $ where
${\bf y}=\{y_{i}\}_{i\in Z}\in Y^{\mathbb{Z}}$ and $\sigma$ is the
left shift on $Y^{\mathbb{Z}}$. Now if $f(x,y)$ is an $R$-invariant
measurable function on $X\times Y^{\mathbb{Z}}$ then for every $x\in
X$ the function $f_x({\bf y})=f(x,{\bf y} )$ is a $\sigma$-invariant
function on the Bernoulli $\mathbb{Z}$-system$
(Y^{\mathbb{Z}},\lambda_x^{\mathbb{Z}},\sigma)$, hence a constant,
$\lambda_x^{\mathbb{Z}}$-a.e.; i.e., $f(x,{\bf y})=f(x)$,
$\lambda_\infty$-a.e.. Thus every $R$-invariant function on $X\times
Y^{\mathbb{Z}}$ is $\mathcal{B}\times Y^{\mathbb{Z}}$-measurable.

For any $F\in \mathcal{D}^\mathbb{Z}$ with $\nu_\infty
(\sigma^{-1}F\Delta F)=0$, let $f(x,{\bf y})=1_F({\bf y})$ for
$\lambda_\infty$-a.e. $(x,{\bf y})\in X\times Y^{\mathbb{Z}}$. Then $f$ is
$R$-invariant and so $f$ is $\mathcal{B}\times
Y^{\mathbb{Z}}$-measurable. Thus there exists $E\in \mathcal{B}$
such that $f(x,y)=1_E(x)$ for $\lambda_\infty$-a.e. $(x,y)\in
X\times Y^{\mathbb{Z}}$ since $f$ is a characteristic function. This
implies $F\in \mathcal{F}=\{ \emptyset, Y^{\mathbb{Z}}\} \, (\bmod
\nu_\infty )$ so $\nu_\infty(F)=0$ or $1$. Hence
$(Y^{\mathbb{Z}},\mathcal{D}^\mathbb{Z},\nu_\infty,\sigma)$ is
ergodic.

Moreover since
$(Y^{\mathbb{Z}},\mathcal{D}^\mathbb{Z},\lambda_x^{\mathbb{Z}},\sigma)$
is ergodic for $\mu$-a.e. $x\in X$ and $\nu_\infty=\int_X
\lambda_x^{\mathbb{Z}} d\mu(x)$, we have
$\lambda_x^{\mathbb{Z}}=\nu_\infty$ for $\mu$-a.e. $x\in X$.
Considering the projection of zero coordinate, $\lambda_x=\nu$ for
$\mu$-a.e. $x\in X$. Hence $\lambda=\mu\times \nu$. Then it follows
that $(X,\mathcal{B},\mu,T)$ is disjoint from
$(Y,\mathcal{D},\nu,S)$.
\end{proof}

The following result is also obvious.

\begin{theorem} \label{thm-lift-key} Let $\pi:(X,\mathcal{B},\mu,T)\rightarrow (Y,\mathcal{D},\nu,S)$ be a factor map
between two MDS's. Then $Dims_\mu(X,T)\supseteq Dims_\nu(Y,S)$. In
particular, the dimension set is invariant under measurable
isomorphism, and so is the entropy dimension.
\end{theorem}

In the following, we consider some special case for dimension set.

\begin{definition}\label{definition-ud}
Let $\tau\in (0,1]$, we call $(X,\mathcal{B},\mu,T)$ a {\it
$\tau-$uniform entropy dimension system} ($\tau-$u.d. system for
short) if $Dims_\mu(X,T)=\{\tau\}$ and call $(X,\mathcal{B},\mu,T)$
a {\it $\tau^+-$ dimension system} ($\tau^+-$d. system for short) if
$Dims_\mu(X,T)\subset [\tau,1]$. If $0\notin Dims_\mu (X,T)$, we
will say $(X,\mathcal{B},\mu,T)$ has {\it strictly positive entropy
dimension}.
\end{definition}

The motivation that we consider the u.d. systems comes from the
K-mixing systems. We can view the u.d. systems as the analogy of the
K-mixing properties in zero entropy situation.

By Definition \ref{definition-ud} and Theorem \ref{thm-lift-key} we
have
\begin{proposition}\label{lift} Let $\tau \in (0,1]$. Then
\begin{enumerate}
\item
A nontrival factor of a $\tau-$u.d. system is also a $\tau-$u.d.
system.
\item
A nontrival factor of a $\tau^+-$d. system is also a $\tau^+-$d.
system.

\item If a system has strictly positive entropy dimension, then any nontrival factor of this system  also has
strictly positive entropy dimension.
\end{enumerate}
\end{proposition}

\begin{lemma} \label{lem-udwm} Let $(X,\mathcal{B},\mu,T)$ be a MDS. If $(X,\mathcal{B},\mu,T)$ has strictly positive entropy
dimension, then $(X,\mathcal{B},\mu,T)$ is weakly mixing.
\end{lemma}
\begin{proof} It is well known that if $(X,\mathcal{B},\mu,T)$ is not weakly mixing,
then there exists a nontrivial factor $(Y,\mathcal{D},\nu,S)$
of $(X,\mathcal{B},\mu,T)$ with discrete spectrum. By Kushnirenko \cite{Ku}, $(Y,\mathcal{D},\nu,S)$ is null. By Theorem \ref{null-zero},
$\overline{D}_\nu(Y,S)=0$, a contradiction with Proposition \ref{lift} (3).
\end{proof}
As a direct application of Theorem \ref{dis}, we have
\begin{corollary} \label{cor-1} 1. $\alpha-$u.d. invertible MDS's are disjoint from ergodic $\beta-$u.d. MDS's
when $1\ge \alpha>\beta\ge 0$.

2. An invertible MDS  which has strictly positive entropy dimension
is disjoint from all ergodic $0$-entropy dimension MDS's.
\end{corollary}

The following example
shows that two systems with the same entropy dimension may also have
disjointness property.
\begin{example}
Choose $0<r_i\nearrow1$ and let $(X_i,\mathcal B_i,\mu_i,T_i)$ be
$r_i-$u.d. invertible MDS(in section 5 we will show  existence of
such MDS's 
). Let $(X,\mathcal B,\mu,T)$ be the product system, i.e.
$(X,\mathcal
B,\mu,T)=(\prod_{i=1}^{\infty}X_i,\prod_{i=1}^{\infty}\mathcal
B_i,\prod_{i=1}^{\infty}\mu_i,\prod_{i=1}^{\infty}T_i)$. Then
$\overline{D}_\mu(X,T)=1$ by Theorem \ref{prop-basic}. Also, since each
$(X_i,\mathcal B_i,\mu_i,T_i)$ is weakly mixing, so is $(X,\mathcal
B,\mu,T)$, and hence ergodic. Since $h_\mu(X,T)=0$, any
K-automorphism MDS (which is also a 1-u.d. system) is disjoint from the ergodic MDS $(X,\mathcal
B,\mu,T)$.
\end{example}

\section{Metric entropy dimension of an open cover}

By a {\it topological dynamical system} (TDS for short) $(X,T)$ we
mean a compact metrizable space $X$ together with a surjective
continuous map $T$ from $X$ to itself. Let $(X,T)$ be a TDS and
$\mu\in M(X,T)$, where $M(X,T)$ denotes the collection of invariant
probability measures of $(X,T)$. Denote by $\mathcal{C}_X$ the set
of finite covers of $X$ and $\mathcal{C}^o_X$ the set of finite open
covers of $X$. For a $\mathcal{U}\in \mathcal{C}_X$, we define
$$H_\mu(\mathcal{U})=\inf \{ H_\mu(\alpha):\alpha\in \mathcal{P}_X
\text{ and } \alpha\succeq \mathcal{U}\},$$ where by $\alpha\succeq
\mathcal{U}$ we mean that every atom of $\alpha$ is contained in one
of the elements of $\mathcal U$. We say an increasing sequence
$S=\{s_1<s_2<\cdots\}$ of $\N$ is an {\it entropy generating
sequence} of $\mathcal{U}$ w.r.t. $\mu$ if
$$\liminf_{n\rightarrow\infty}\frac{1}{n}H_\mu(\bigvee_{i=1}^nT^{-s_i}\mathcal{U})>0.$$
We say $S=\{s_1<s_2<\cdots\}$ of $\N$ is a {\it positive entropy
sequence} of $\mathcal{U}$  w.r.t. $\mu$ if
$$h^S_{\mu}(T,\mathcal{U}):=
\limsup_{n\rightarrow\infty}\frac{1}{n} H_\mu
(\bigvee_{i=1}^nT^{-s_i}\mathcal{U})>0.$$

Denote ${\cal E}_\mu(T,\mathcal{U})$ by the set of all entropy
generating sequences of $\mathcal{U}$, and
$\mathcal{P}_\mu(T,\mathcal{U})$ by the set of all positive
entropy sequences of $\mathcal{U}$. Clearly
$\mathcal{P}_\mu(T,\mathcal{U})\supset
\mathcal{E}_\mu(T,\mathcal{U})$.

\begin{definition} \label{c-de-s-1} Let $(X,T)$ be a TDS, $\mu\in
M(X,T)$ and $\mathcal{U}\in \mathcal{C}_X$. We define
$$\overline{D}^e_\mu(T,\mathcal{U})=\begin{cases} \sup \limits_{S\in \mathcal{E}_\mu(T,\mathcal{U})}
\overline{D}(S)\, &\text{if } \mathcal{E}_\mu(T,\mathcal{U})\neq\emptyset\\
0 \, &\text{if } \mathcal{E}_\mu(T,\mathcal{U})=\emptyset
\end{cases},$$
$$\overline{D}^p_\mu(T,\U)=\begin{cases} \sup \limits_{S\in \mathcal{P}_\mu(T,\mathcal{U})} \overline{D}(S)\,
&\text{if } \mathcal{P}_\mu(T,\mathcal{U})\neq \emptyset\\ 0 \,
&\text{if } \mathcal{P}_\mu(T,\mathcal{U})=\emptyset
\end{cases}.$$
Similarly, we can define $\underline{D}^e_\mu(T,\U)$ and
$\underline{D}^p_\mu(T,\alpha)$ by changing the upper dimension into
the lower dimension.
\end{definition}

Similar to Proposition \ref{prop-kkk} we have
\begin{proposition}\label{c-prop-re} Let $(X,T)$ be a TDS, $\mu\in
M(X,T)$ and $\mathcal{U}\in \mathcal{C}_X$. Then
$$\underline{D}^e_\mu(T,\mathcal{U})\le \overline{D}^e_\mu(T,\mathcal{U})= \underline{D}^p_\mu(T,\mathcal{U})\le
\overline{D}^p_\mu(T,\mathcal{U}).$$
\end{proposition}

By Proposition \ref{c-prop-re}, we have
\begin{definition} \label{c-de-s-3} Let $(X,T)$ be a TDS, $\mu\in
M(X,T)$ and $\mathcal{U}\in \mathcal{C}_X$. We define
$$\overline{D}_\mu(T,\mathcal{U}):=\overline{D}^e_\mu(T,\mathcal{U})= \underline{D}^p_\mu(T,\mathcal{U}),$$
which is called the upper entropy dimension of $\mathcal{U}$. Similarly, we have the definition of lower dimension and dimension.
\end{definition}

\begin{theorem}\label{thm-s-4} Let $(X,T)$ be a TDS and $\mu\in
M(X,T)$. Then
$$\overline{D}_\mu(X,T)=\sup_{\mathcal{U} \in \mathcal{C}^o_X}
\overline{D}_\mu(T,\mathcal{U}),$$
where $\mathcal{C}^o_X$ is the set of finite open
covers of $X$.
\end{theorem}
\begin{proof} Let $\mathcal{U}=\{ U_1,U_2,\cdots,U_n\}\in \mathcal{C}_X^o$. For any
$s=(s(1),\cdots,s(n))\in { \{ 0,1 \} }^n$, set
$U_s=\bigcap_{i=1}^{n}U_i(s(i))$, where $U_i(0)=U_i$ and
$U_i(1)=X\setminus U_i$. Let $\alpha=\{ U_s: s\in { \{ 0,1 \} }^n
\}$. Then $\alpha$ is the Borel partition generated by $\mathcal{U}$
and $\overline{D}_\mu(X,T)\ge \overline{D}_\mu(T,\alpha)\ge \overline{D}_\mu(T,\mathcal{U})$. Since
$\mathcal{U}$ is arbitrary, we get $\overline{D}_\mu(X,T)\ge \sup_{\mathcal{U}
\in \mathcal{C}^o_X} \overline{D}_\mu(T,\mathcal{U})$.

For the other direction, let $\alpha=\{ A_1,\cdots,A_k \}\in
\mathcal{P}_X$. If $\overline{D}_\mu(T,\alpha)=0$, it is obvious
$\overline{D}_\mu(T,\alpha)\le \sup_{\mathcal{U} \in \mathcal{C}^o_X}
\overline{D}_\mu(T,\mathcal{U})$. Now assume that $\overline{D}_\mu(T,\alpha)>0$. For
any $\epsilon>0$, there exists $S=\{ s_1<s_2<s_3<\cdots\} \in
\mathcal{P}_\mu(T,\alpha)$ with
$\underline{D}(S)>\overline{D}_\mu(T,\alpha)-\epsilon$.

Let $a:=\frac{h_\mu^S(T,\alpha)}{2}>0$. We have

\noindent{\bf Claim.} There exists $\mathcal{U}\in \mathcal{C}_X^{
o}$ such that $H_{\mu} (T^{-i}\alpha| \beta )\le a$ if $i\in
\mathbb{Z}_+$ and $\beta\in \mathcal{P}_X$ satisfying $\beta\succeq
T^{-i} \mathcal{U}$.

\begin{proof}[Proof of Claim]
By \cite[Lemma 4.15]{Wal}, there exists $\delta_1= \delta_1 (k,
\epsilon)>0$ such that if
 $\beta_i=\{ B^i_1,\cdots,B^i_k \}\in \mathcal{P}_X, i= 1, 2$ satisfy
 $\sum_{i=1}^k
\mu(B^1_i\Delta B^2_i)<\delta_1$ then $H_{\mu}
(\beta_1|\beta_2)\le a$. Since $\mu$ is regular, we can take
closed subsets $B_i\subseteq A_i$ with $\mu(A_i\setminus
B_i)<\frac{\delta_1}{2k^2}$, $i=1,\cdots,k$. Let $B_0=X\setminus
\bigcup_{i=1}^kB_i$, $U_i=B_0\cup B_i,i=1,\cdots,k$. Then
$\mu(B_0)<\frac{\delta_1}{2k}$ and $\mathcal{U}=\{ U_1,\cdots,U_k
\}\in \mathcal{C}^{o}_X$.

Let $i\in \mathbb{Z}_+$. If $\beta\in \mathcal{P}_X$ is finer than
$T^{-i} \mathcal{U}$, then we can find $\beta' =\{ C_1,\cdots,C_k
\}\in \mathcal{P}_X$ satisfying
 $C_j\subseteq T^{-i} U_j, \ j=1,\cdots,k$ and $\beta\succeq
\beta'$, and so $H_{\mu} (T^{-i}\alpha| \beta)\le H_{\mu}
(T^{-j}\alpha| \beta')$. For each $j=1, \cdots, k$, as
$T^{-i}U_j\supseteq C_j \supseteq X\setminus \bigcup_{l\neq j}
T^{-i}U_l=T^{-i}B_j$ and $T^{-i}A_j\supseteq T^{-i}B_j$, one has
\begin{align*}
\mu(C_j\Delta T^{-i}A_j)&\le \mu(T^{-i}A_j\setminus
T^{-i}B_j)+\mu(T^{-i}B_0)=\mu(A_j\setminus
B_j)+\mu(B_0)\\
&<\frac{\delta_1}{2k}+\frac{\delta_1}{2k^2}\le \frac{\delta_1}{k}.
\end{align*}
Thus $\sum \limits_{j=1}^k \mu(C_j\Delta T^{-i}A_j)<\delta_1$.  It
follows that $H_{\mu}(T^{-i}\alpha|\beta')\le a$ and so
$H_{\mu}(T^{-i}\alpha|\beta)\le a$.
\end{proof}

For $n\in \mathbb{N}$, if $\beta_n\in \mathcal{P}_X$ with
$\beta_n\succeq \bigvee \limits_{i=1}^n T^{-s_i}\mathcal{U}$ then
$\beta_n\succeq T^{-s_i}\mathcal{U}$ for each $i\in \{
1,2,\cdots,n\}$, and so using the above Claim one has
\begin{align*}
H_{\mu}(\bigvee_{i=1}^n T^{-s_i}\alpha) &\le H_{\mu}(\beta_n)+
H_{\mu}(\bigvee_{i=1}^n T^{-s_i}\alpha|\beta_n) \\
&\le H_{\mu}(\beta_n)+\sum_{i=1}^nH_{\mu}(T^{-s_i}\alpha|\beta_n)
\le H_{\mu}(\beta_n)+na.
\end{align*}
Moreover, since the above inequality is true for any $\beta_n\in
\mathcal{P}_X$ with $\beta_n\succeq \bigvee \limits_{i=1}^n
T^{-s_i}\mathcal{U}$, one has $H_\mu (\bigvee \limits_{i=1}^n
T^{-s_i}\alpha)\le H_\mu (\bigvee \limits_{i=1}^n
T^{-s_i}\mathcal{U})+ na$. Recall that $h_\mu^S(T,\alpha)=2a$, one has $h_\mu^S(T,\mathcal{U})\ge a>0$.
Also recall that $\overline{D}_\mu(T,\alpha)= \underline{D}^p_\mu(T,\alpha)$ (Definition \ref{de-s-3}),
we then have
$$\overline{D}_\mu(T,\mathcal{U})\ge \underline{D}(S)\ge \underline{D}^p_\mu(T,\alpha)-\epsilon
=\overline{D}_\mu(T,\alpha)-\epsilon.$$ This implies $\sup_{\mathcal{U} \in
\mathcal{C}^o_X} \overline{D}_\mu(T,\mathcal{U})\ge \overline{D}_\mu(T,\alpha)-\epsilon$.
Finally, since $\alpha$ and $\epsilon$ are arbitrary, we get
$\sup_{\mathcal{U} \in \mathcal{C}^o_X} \overline{D}_\mu(T,\mathcal{U})\ge
\overline{D}_\mu(X,T)$.
\end{proof}

Now let us recall the corresponding notions in topological
setttings, which appeared in \cite{DHP}.

Let $(X,T)$ be a TDS and $\U\in \mathcal{C}_X^o$. We recall
\cite{DHP} that an increasing sequence $S=\{s_1<s_2<\cdots\}$ of
$\N$ is an {\it entropy generating sequence} of $\U$ if
$$\liminf_{n\rightarrow\infty}\frac{1}{n}\log \mathcal N(\bigvee_{i=1}^nT^{-s_i}\U)>0,$$
and $S=\{s_1<s_2<\cdots\}$ of $\N$ is a {\it positive entropy
sequence} of $\U$ if
$$h^S_{\text{top}}(T,\mathcal{U}):=
\limsup_{n\rightarrow\infty}\frac{1}{n}\log
\mathcal N(\bigvee_{i=1}^nT^{-s_i}\U)>0.$$
Here $\mathcal N(\mathcal U)$ is the number of the sets in a subcover of $\mathcal U$ with smallest cardinality.

Denote by ${\cal E}(T,\U)$ the set of all entropy generating
sequences of $\U$, and by $\mathcal{P}(T,\mathcal{U})$ the set of
all positive entropy sequences of $\U$. Clearly
$\mathcal{P}(T,\mathcal{U})\supset \mathcal{E}(T,\mathcal{U})$.

Let $(X,T)$ be a TDS and $\mathcal{U}\in \mathcal{C}_X^o$. We define
$$\overline{D}_e(T,\U)=\begin{cases} \sup_{S\in \mathcal{E}(T,\U)} \overline{D}(
S)\, &\text{if } \mathcal{E}(T,\U)\neq \emptyset\\ 0 \, &\text{if
} \mathcal{E}(T,\U)=\emptyset \end{cases},$$
$$\underline{D}_p(T,\U)=\begin{cases} \sup_{S\in \mathcal{P}(T,\U)}
\underline{D}(S)\, &\text{if } \mathcal{P}(T,\U)\neq\emptyset\\
0 \, &\text{if } \mathcal{P}(T,\U)=\emptyset \end{cases}.$$

It is similar to the proof of Proposition \ref{prop-kkk}, we have
$\overline{D}_e(T,\U)=\underline{D}_p(T,\U)$ for any
$\mathcal{U}\in \mathcal{C}_X^o$. Hence we define
$$\overline{D}(T,\mathcal{U}):=\overline{D}_e(T,\U)=\underline{D}_p(T,\U) \text{ and}$$
$$\overline{D}(X,T)=\sup_{\U \in \mathcal{C}_X^0}
\overline{D}(T,\mathcal{U}).$$ We call $\overline{D}(X,T)$ the upper entropy dimension of
$(X,T)$. Similarly, we have the definition of lower dimension and dimension.

Using Theorem \ref{thm-s-4}, we have
\begin{theorem}  \label{thm-ine}Let $(X,T)$ be a TDS and $\mu\in
M(X,T)$. Then
$$\overline{D}_\mu(X,T)\le \overline{D}(X,T).$$
\end{theorem}


\begin{example} For any $\tau\in (0,1]$, there exists a minimal system $(X,T)$ satisfying
$D(X,T)=\tau$ and $D_\mu(X,T)=0$ for any $\mu\in M(X,T)$.
\end{example}
\begin{proof}
Let $(X,T)$ be the system generated by Cassaigne's model \cite{C}(the
uniformly recurrent one), then it is minimal. By taking
$\phi(n)=\frac{n}{\log n}$ in this construction, we get $D(X,T)=1$.
By taking $\phi(n)=n^{\tau}$ in this construction, we get
$D(X,T)=\tau$, for any $0<\tau<1$. In \cite{ADP}, it is shown that
$(X,T)$ is uniquely ergodic and with respect to the unique ergodic invariant
measure $\mu$, $D_\mu(X,T)=0$.
\end{proof}

\begin{definition} An invertible TDS $(X,T)$ is {\it doubly minimal} if for all
$x,y\in X$, $y\in \{ T^nx\}_{n\in \mathbb{Z}}$, $\{
(T^jx,T^jy)\}_{j\in \mathbb{Z}}$ is dense in $X\times X$.
\end{definition}

The following results is Theorem 5 in \cite{W}.
\begin{lemma} \label{lem-expre} Any ergodic system $(Y,\mathcal{C},\nu,S)$ with
$h_\nu(S)=0$ has a uniquely ergodic topological model $(X,T)$ that
is doubly minimal.
\end{lemma}

It is well known that any doubly minimal system has zero entropy
(see \cite{W}). However we have
\begin{example} There exists a doubly minimal system with positive
entropy dimension.
\end{example}
\begin{proof} This comes directly from Lemma \ref{lem-expre} and
Theorem \ref{thm-ine} since there exists an ergodic system with
metric entropy dimension $0<\tau<1$(see section 5).
\end{proof}

 A TDS $(X,T)$ with metric $d$ is called {\it  distal}, if
$\inf_{n\ge 0} d(T^nx,T^ny)>0$
 for every $x\neq y\in X$. Let $(X,\mathcal{B},\mu,T)$ be an invertible ergodic MDS.  A
sequence $A_1\supset A_2\supset A_3\cdots$ of sets in $\mathcal{B}$
with $\mu(A_n)>0$ and $\mu(A_n)\rightarrow  0$, is called {\it a
separating sieve} if there exists a subset $X_0\subset X$ with
$\mu(X_0) = 1$ such that for every $x, x'\in X_0$ the condition for
every $n\in \mathbb{N}$ there exists $k\in \mathbb{Z}$ with $T^kx,
T^kx'\in A_n$ implies $x = x'$. We say that the invertible ergodic
MDS $(X,\mathcal{B},\mu,T)$ is {\it measure distal} if either
$(X,\mathcal{B},\mu,T)$ is finite or there exists a separating
sieve. In \cite{Lin} E. Lindenstrauss shows that every invertible
ergodic measure distal MDS can be represented as a minimal
topologically distal system.

It is well known that a distal TDS has zero topological entropy, and
an invertible ergodic measure distal MDS has zero measure entropy.
  To end this section let us ask the
following questions:

\begin{question}
\begin{enumerate}
\item Is the entropy dimension of a minimal
distal TDS zero?
\item Is the entropy dimension of an invertible
ergodic measure distal MDS zero?
\end{enumerate}
\end{question}
\section{The existence of u.d. MDS's}
In this section, our aim is to show that for every $\tau\in (0,1)$,
there exists a MDS $(X,\mathcal{B},\mu,T)$ having the property of
$\tau$-u.d.. We mention that a $K-$mixing system is of u.d. for $\tau=1$ and an irrational rotation is of u.d. for $\tau=0$.

Our construction employs the so called ``cutting and stacking" method.
Let $X$ be the interval $[0,1)$, $\mathcal {B}$ be the Borel $\sigma$-algebra on $[0,1)$ and $\mu$ be the Lebesgue measure on $[0,1)$.
In the cutting and stacking construction, $[0,1)$ will be cut into many subintervals and all of them are left closed and right open.
Let $B_i\subset [0,1), 1\le i\le h$, be $h$ disjoint subintervals of the same length. A {\it column} $C$ is the ordered set of these intervals,
i.e. $$C=\{B_1,B_2,\cdots,B_h\}=\{B_i: 1\le i\le h\}.$$
We can consider $C$ obtained by ``stacking" the $B_i$'s one by one. We say the column $C$ has {\it base} $B_1$, {\it top} $B_h$, {\it height} $h(C)=h$ and
{\it width} $w(C)=\text{ the length of }B_i$. Denote $|C|=\cup_{i=1}^hB_i$.
We call each $B_i$ a {\it level set} of $C$. For the column $C$, the map $T$ will map $B_i$
linearly onto $B_{i+1}$ for $1\le i \le n-1$, and is undefined on $B_n$. We call $C^0=\{B_1^0,TB_1^0,\cdots,T^{h-1}B_1^0\}$ a {\it subcolumn} of $C$ if $B_1^0\subset B_1$.
A {\it tower} $W$ is a finite collection of columns, which generally have different heights.
In this paper, all the columns of a tower will have the same height. The width of tower $W$ is $$w(W)=\sum\limits_{C \text{ is a column of }W} w(C).$$
The {\it cardinality} of a tower $W$, denoted by $\#W$, is the number of its columns. We denote by $|W|$ the union of all the level sets of its columns.
The {\it base} of the tower $W$, which is denoted by $base(W)$, is the union of all the bases of its columns. For the tower $W$, $T$ is considered to be
defined on each of its columns except on the tops of its columns. Hence for the tower $W$, $T$ is undefined on a set of measure $w(W)$.
The cutting and stacking method will construct a sequence of towers with widths tending to $0$ so that $T$ is invertible on the whole interval $[0,1)$ up to
a set of measure zero. (One may see \cite{Friedman,S1,S2} for the basics of the cutting and stacking method.)

In our construction, we divide $[0,1)$ into three parts:
$P_0=[0,\frac{\xi}{2})$, $P_1=[\frac{\xi}{2},\xi)$ and
$P_s=[\xi,1)$, where we will decide $\xi$ later (see \eqref{xi} in section 5.4) and ``s'' stands for ``spacer''. This will be our initial tower and any level set
of any other tower will be a subset of $P_0, P_1,\text{ or } P_s$. Due to the initial tower, we say that a level set $B$ has a {\it name} ``a'' if
$B\subset P_a, a=0,1,s$. The name of a column $C=\{B_1,B_2,\cdots,B_h\}$ is a word $b=b_1b_2\cdots b_h\in \{0,1,s\}^h$, where $b_i$ is the name of $B_i$.
The name of a tower is the collection of names of its columns. By $N(W)$ we denote the number of different names of columns of the tower $W$.
We say two columns are {\it isomorphic} if they have the same name (they don't need to have the same widths). We say two towers $W=\{C_1,C_2,\cdots,C_k\}$ and $W'=\{C_1',C_2',\cdots,C_k'\}$ with the same cardinality
are {\it isomorphic} if after some reordering, the columns $C_i$ and $C_i'$ are isomorphic and $w(C_i)=\lambda w(C_i')$ for some $\lambda>0$ and all $1\le i \le k$.
If $C^0$ is a subcolumn of a column $C$ of the tower $W$, then we may say that $C^0$ is a subcolumn of $W$.
A {\it segment} $S$ of height $\ell$ of a column $C=\{B_1,B_2,\cdots,B_h\}$ is a collection of consecutive level sets $\{B_{\ell'},B_{\ell'+1},\cdots,B_{\ell'+\ell-1}\}$
starting at some position $\ell'$ with $1\le \ell' \le h-\ell+1$. If $b=b_1b_2\cdots b_h\in \{0,1,s\}^h$ is the name of the column $C$, then the word $b_{\ell'}b_{\ell'+1}\cdots b_{\ell'+\ell-1}$ is the name of the segment $S$. We also write $S$ as $S=B_{\ell'}B_{\ell'+1}\cdots B_{\ell'+\ell-1}$ or $S=b_{\ell'}b_{\ell'+1}\cdots b_{\ell'+\ell-1}$. Let
$W$ and $W'$ be two towers such that the height of $W'$ is bigger than that of $W$. We say a segment $S$ of some column of $W'$ is a $W-$segment if
the name of $S$ is the same as the name of some column $C$ of $W$ (in this case, we also say $S$ is isomorphic to $C$).

Through cutting and stacking steps, we successively construct a sequence of towers to get a MDS
with a given entropy dimension $\tau\in (0,1)$, by controlling the heights of independent and repetition steps.
We can see clearly from the construction what the entropy generating sequence is. We use three types of operations which will be described in the next section.

\subsection{Three kinds of operations.}\ \

Now we will describe the three kinds of operations we need.

{\bf 1. Independent cutting and stacking.}

Let $W^1$ and $W^2$  be two towers with the same width $w$. Assume $W^j$ has $c_j-$many columns $C^j_1,C^j_2,\cdots,C^j_{c_j}$ for $j=1,2$. We divide
each column of $W^1$ into subcolumns according to the distribution of the columns of $W^2$. That is,
we divide each column $C^1_i$ into $c_2-$many subcolumns $C^1_{i,k}$ with width $w(C^1_{i,k})=w(C^1_i)\frac{w(C^2_k)}{w}$, $i=1,2,\cdots, c_1, k=1,2,\cdots,c_2$.
Likewise we divide each column $C^2_k$ into $c_1-$many subcolumns $C^2_{k,i}$ with width $w(C^2_{k,i})=w(C^2_k)\frac{w(C^1_i)}{w}$, $i=1,2,\cdots, c_1, k=1,2,\cdots,c_2$.
Since we have $w(C^1_{i,k})=w(C^2_{k,i})$, we stack each $C^2_{k,i}$ on top of $C^1_{i,k}$ to form a new column $C^1_{i,k}*C^2_{k,i}$.
Denote the new tower $\{C^1_{i,k}*C^2_{k,i}\}$ by $W^1*_{ind}W^2$.

For a tower $W$ and an integer $e\ge 1$, we equally divide $W$ into $e-$many subtower $W^1,W^2,\cdots,W^e$ (we divide each column of $W$ into $e-$many subcolumns equally and take all the $i-$th subcolumn to make the tower $W^i$). We call the tower
$Ind(W,e)=W^1*_{ind}W^2*_{ind}\cdots*_{ind}W^e$ the $e-$many independent cutting and stacking of $W$. We note that
$\#Ind(W,e)=(\#W)^e, h(Ind(W,e))=eh(W)$.
In fact we can cut each column of $W$ into $e(\#W)^{e-1}-$many subcolumns equally and then choose these subcolumns from different $e-$many combinations
of columns of $W$ to stack to form $Ind(W,e)$, i.e., the tower $Ind(W,e)$ is stacked by $e-$many $W$-segments independently.

{\bf 2. Repetitive cutting and stacking.}

For a tower $W=\{C_1,C_2,\cdots,C_c\}$ and an integer $r\ge 1$, we equally divide each column $C_i$ of $W$ into $r-$many subcolumns
$C_{i,1},C_{i,2},\cdots,C_{i,r}$ and stack them one by one to make a new column $C_{i,1}*C_{i,2}*\cdots*C_{i,r}$. Then we call the tower
$Rep(W,r)=\{C_{i,1}*C_{i_2}*\cdots*C_{i,r}:i=1,2,\cdots,c\}$ the $r-$many repetitive cutting and stacking of $W$. We note that $\#Rep(W,r)=\#W$.

{\bf 3. Inserting spacers while independent cutting and stacking.}

Let $W$ be a tower with columns $\{C_1,C_2,\cdots,C_c\}$ and $e,h^*\ge 1$ be two integers. Due to the definition of $Ind(W,e)$,
we can assume that the tower $Ind(W,e)$ is formed by columns $\overline{C}_{i_1}*\overline{C}_{i_2}*\cdots *\overline{C}_{i_e}$
for $i_1,i_2,\cdots,i_e\in\{1,2,\cdots,c\}$, where $\overline{C}_i$ is a subcolumn of $C_i$.
Cut each column
$\overline{C}_{i_1}*\overline{C}_{i_2}*\cdots *\overline{C}_{i_e}$ of $Ind(W,e)$ into $c-$many subcolumns equally, which we denote by
$(\overline{C}_{i_1}*\overline{C}_{i_2}*\cdots *\overline{C}_{i_e})_{i_{e+1}}$, $i_{e+1}=1,2,\cdots,c$.
Since $Ind(W,e)$ has $c^e$-many columns, the new tower has $c^{e+1}$-many columns. We write $(\overline{C}_{i_1}*\overline{C}_{i_2}*\cdots *\overline{C}_{i_e})_{i_{e+1}}$ as $\overline{C}_{i_1}^{i_{e+1}}*\overline{C}_{i_2}^{i_{e+1}}*\cdots *\overline{C}_{i_e}^{i_{e+1}}$ and call $\overline{C}_{i_k}^{i_{e+1}}$ the $k$-th $W$-segment of the column
$\overline{C}_{i_1}^{i_{e+1}}*\overline{C}_{i_2}^{i_{e+1}}*\cdots *\overline{C}_{i_e}^{i_{e+1}}$. Note that $\overline{C}_{i_k}^{i_{e+1}}$ is isomorphic to $C_i$.
Now we will insert $e \cdot h^*-$many spacers altogether between these $W$-segments
of $\overline{C}_{i_1}^{i_{e+1}}*\overline{C}_{i_2}^{i_{e+1}}*\cdots *\overline{C}_{i_e}^{i_{e+1}}$, where each spacer is an interval of length $w(\overline{C}_{i_1}^{i_{e+1}}*\overline{C}_{i_2}^{i_{e+1}}*\cdots *\overline{C}_{i_e}^{i_{e+1}})$ cut from $P_s$.
For $k=1,2,\cdots,e$, let $\ell=i_{k+1}(mod\ \ h^*), 0\le\ell\le h^*-1$, we insert $\ell-$many spacers before the $k-$th $W$-segment $\overline{C}_{i_k}^{i_{e+1}}$ and $(h^*-\ell)-$many spacers after.
That is, we change each $\overline{C}_{i_k}^{i_{e+1}}$ into $s^{\ell}\overline{C}_{i_k}^{i_{e+1}}s^{h^*-\ell}$ (here we identify the inserted spacer with its name ``$s$'').
Denote the new tower by $Inds(W,e,h^*)$. Each column of $Inds(W,e,h^*)$ is formed by $e-$many such segments of the form $s^{\ell}\overline{C}_{i_k}^{i_{e+1}}s^{h^*-\ell}$. We should notice here that some columns of $Inds(W,e,h^*)$ may have the same name. Furthermore,
\begin{equation}\label{equation-name}
(N(W))^e\le N(Inds(W,e,h^*))\le \#Inds(W,e,h^*)=(\#W)^{e+1}.
\end{equation}

We write the new tower $Inds(W,e,h^*)$ by $\overline {W}$. For convenience, we still call $s^{\ell}\overline{C}_{i_k}s^{h^*-\ell}$ the $k$-th
$W$-segment of the column of $\overline {W}$ or simply the $k$-th $W$-segment of $\overline {W}$ ignoring the spacers.

If we denote by $p_{\ell}$ the probability of all the columns of $\overline{W}$
whose $k$-th $W$-segment begins with $\ell$-many spacers, then $p_{\ell}$ is
either $\frac{\lfloor\frac{c}{h^*}\rfloor}{c}$ or
$\frac{\lfloor\frac{c}{h^*}\rfloor+1}{c}$, independent of $k$ and $\ell$.
Since $$\frac{|p_{\ell}-\frac{1}{h^*}|}{\frac{1}{h^*}}\le\frac{h^*}{c},$$ we may say that the number of beginning spacers of the $k$-th $W$-segment of $\overline{W}$
is uniformly distributed on $\{0,1,\cdots,h^*-1\}$ within $\frac{h^*}{c}-$error.

\subsection{The choice of the parameters.}\ \

To construct a MDS $(X,\mathcal{B},\mu,T)$ with $\tau$-u.d. for fixed $\tau\in (0,1)$, we need to define sequences of
integer parameters $1<r_1<r_2<\cdots,\, 1<e_0, e_1, e_2, \cdots$, $1\le w_0<w_1<w_2<\cdots$, $1\le h_0<h_1<h_2<\cdots$
and $1< \tilde h_0<\tilde h_1<\tilde h_2<\cdots$.

Given $\tau\in (0,1)$, we let $r_n=C_\tau n^2$, where $C_\tau$ is an integer such that $C_\tau^{\frac{\tau}{1-\tau}}>4$.
Let $1<l_1<n_1<l_2<n_2\cdots$ be any sequence of integers satisfying that
\begin{align}\label{nt}
  \sum_{t=1}^\infty\frac{1}{(n_t)^{\frac{2\tau}{1-\tau}}}<\infty.
\end{align}
Put
\begin{equation*}
 e_0=2, h_0=1,w_0=1 \text{ and } h_1=e_0.
\end{equation*}
Next we inductively construct $h_n,\tilde h_n, w_n, e_n$. For $n\ge 1$, put
\begin{equation}\label{condition0hw}
 \tilde h_n=h_nr_n,\, w_n=\begin{cases} \tilde h_n &\text{if } n\not\in \{n_1,n_2,\cdots\} \\\tilde h_{n_t}+h_{l_t} \, &\text{if
  } n=n_t \text{ for some } t\end{cases},
\end{equation}
\begin{align}\label{condition0e}
e_n=\lfloor(\frac{(w_n)^\tau}{e_0e_1\cdots e_{n-1}})^{\frac{1}{1-\tau}}\rfloor
\end{align}
and
\begin{align}\label{condition0h}
 h_{n+1}=w_ne_{n}.
\end{align}
Since $r_n\rightarrow +\infty$, it is clear from \eqref{condition0hw}, \eqref{condition0e} and \eqref{condition0h},
\begin{align}\label{condition11}
  \lim \limits_{n\rightarrow +\infty} \frac{w_n}{\tilde h_n}=1.
\end{align}
By \eqref{condition0e}, we have
\begin{align}\label{conditione-r0}
e_{1}&=\lfloor(\frac{(w_1)^\tau}{e_0})^{\frac{1}{1-\tau}}\rfloor=\lfloor(\frac{(h_1r_1)^\tau}{e_0})^{\frac{1}{1-\tau}}\rfloor
=\lfloor(\frac{(2C_{\tau})^\tau}{2})^{\frac{1}{1-\tau}}\rfloor=\lfloor \frac{C_\tau^{\frac{\tau}{1-\tau}}}{2} \rfloor \ge 2
\end{align}
and for $n\ge 2$,
\begin{align}\label{conditione-r1}
e_{n}&\ge \lfloor(\frac{(w_{n-1}e_{n-1}r_{n})^\tau}{e_0e_1\cdots e_{n-1}})^{\frac{1}{1-\tau}}\rfloor
=\lfloor(\frac{(w_{n-1}e_{n-1})^\tau}{e_0e_1\cdots e_{n-1}})^{\frac{1}{1-\tau}} \cdot (r_{n})^{\frac{\tau}{1-\tau}}\rfloor \nonumber\\
&=\lfloor \frac{(\frac{(w_{n-1})^\tau}{e_0e_1\cdots e_{n-2}})^{\frac{1}{1-\tau}}}{ e_{n-1}}\cdot (r_{n})^{\frac{\tau}{1-\tau}}\rfloor
 \ge\lfloor (r_{n})^{\frac{\tau}{1-\tau}}\rfloor.
\end{align}
Hence we have $\lim \limits_{n\rightarrow +\infty}e_n=+\infty$ and by the definition of $e_n$'s,
\begin{align}\label{condition1}
\lim_{n\rightarrow+\infty}\frac{e_0e_1\cdots
e_n}{(w_n e_n)^{\tau}}=1.
\end{align}
Now by \eqref{condition11} and \eqref{condition1}, we have
\begin{align}\label{condition2}
e_n&=\lfloor(\frac{(w_n)^\tau}{e_0e_1\cdots e_{n-1}})^{\frac{1}{1-\tau}}\rfloor=\lfloor (\frac{w_n}{\tilde h_n})^{\frac{\tau}{1-\tau}}\cdot (\frac{(w_{n-1}e_{n-1}r_{n})^\tau}{e_0e_1\cdots e_{n-1}})^{\frac{1}{1-\tau}}\rfloor \nonumber\\
&=\lfloor (\frac{w_n}{\tilde h_n})^{\frac{\tau}{1-\tau}}\cdot \frac{(\frac{(w_{n-1})^\tau}{e_0e_1\cdots e_{n-2}})^{\frac{1}{1-\tau}}}{ e_{n-1}}\cdot (r_{n})^{\frac{\tau}{1-\tau}}\rfloor \\
&\sim (r_n)^{\frac{\tau}{1-\tau}}.\nonumber
\end{align}
From \eqref{conditione-r0}, \eqref{conditione-r1} and the setting $r_n=C_\tau n^2$, we have $e_{n}\ge 2$ for every $n\ge 1$. Together with \eqref{condition0e}, we deduce that
\begin{align}\label{condition3}
w_n\ge (e_0e_1\cdots e_{n-1})^\frac{1}{\tau}\ge 2^{\frac{n}{\tau}}.
\end{align}
Note that from \eqref{condition2} and the setting $r_n=C_\tau n^2$, both $e_n$ and $r_n$ have polynomial growth rate on $n$.
Hence for any $\epsilon>0$, we have that
\begin{align}\label{condition4}
\lim_{n\rightarrow+\infty}\frac{(w_n)^{\epsilon}}{e_n}=\lim_{n\rightarrow+\infty}\frac{(w_n)^{\epsilon}}{(r_n)^{\frac{\tau}{1-\tau}}}=+\infty.
\end{align}
\subsection{The construction.}\ \

Let $W_0=\tilde{W}_0=\{P_0,P_1\}$ be the $0$-th and $\tilde 0$-th step tower. We note here that $W_0$ and $\tilde{W}_0$ do not contain a subset of $P_s$.
The construction consists of a sequence of steps, step $n$ and step $\tilde n$, $n\in\N$. The steps occur in the following order:
Step 1, Step $\tilde 1$, Step 2, Step $\tilde 2$, $\cdots$, Step $n$, Step $\tilde n$, $\cdots$.

At step 1, we do $e_0-$many independent cutting and stacking of $\tilde{W}_0$ to construct the first tower $W_1$ of height $h_1=e_0$, i.e.,
$W_1=Ind(\tilde{W}_0,e_0)$. We have $2^{e_0}-$many columns of all possible sequences of $0$'s and $1$'s as their names of equal width and height in $W_1$.
Suppose after step $n$ we have obtained the tower $W_n$ of height $h_n$. Then at step $\tilde{n}$,
we do $r_n-$many repetitive cutting and stacking of $W_n$,
i.e. if we denote the tower after this step by $\tilde {W_n}$, then $\tilde {W_n}=Rep(W_n,r_n)$. This step could not increase the complexity too much.
At step $(n+1)$, if $n\not \in \{n_1,n_2,\cdots\}$, we do $e_n-$many independent cutting and stacking of $\tilde{W}_n$,
i.e. $W_{n+1}=Ind(\tilde{W_n},e_n)$. If $n= n_t$ for some $t\ge 1$,
we insert the spacers while doing $e_n$-many independent cutting and stacking between the $\tilde W_{n_t}$-segments as mentioned in section 5.1 (with parameter $h^*=h_{l_t}$),
 i.e., we let  $W_{n_t+1}=Inds(\tilde W_{n_t},e_{n_t},h_{l_t})$. Since we need to show that any non-trivial partition $P=\{A,A^c\}$ has the same entropy dimension,
 we need to be careful not to generate some kind of ``rotation'' factor by level sets. This step makes a level set of a previous step spread out ``almost uniformly''
 to the level sets of future steps.

Then we get an invertible MDS and denote it by
$(X,\mathcal{B},\mu,T)$, where $\mu$ is the Lebesgue measure on $X$, $\mathcal{B}$ is the $\sigma-$algebra
of $X$ generated by the level sets of the sequence of towers $W_{n}$ and $T$ is the associated map.

\subsection{List of the parameters and notations}\ \

We remind the following parameters and notations.
\begin{itemize}
  \item $e_n$--- we do $e_n-$many independent cutting and stacking at Step $n+1$ if $n\notin \{n_1,n_2,\cdots\}$; we insert spacers while
  doing $e_n-$many independent cutting and stacking at Step $n+1$ if $n\in \{n_1,n_2,\cdots\}$.
  \item $r_n$--- we do $r_n-$many repetitive cutting and stacking at Step $\tilde n$.
  \item $W_n$ and $\tilde W_n$--- towers after Step $n$ and Step $\tilde n$, respectively.
  \item $W_n$-segments and $\tilde W_n$-segments--- subcolumns of columns of the towers $W_n$ and $\tilde W_n$ respectively, when seeing from towers in further steps.
  If $n=n_t$ for some $t$, a $\tilde W_n$-segment $S$ together with the adding spacers, which has the form $s^{\ell}Ss^{h_{\ell_t}-\ell}$, is also called a $\tilde W_n$-segment.
  \item $h_n$ and $\tilde h_n$--- height of columns of the tower $W_n$ and $\tilde W_n$, respectively.
  \item $w_n$--- height of $\tilde W_n$-segments when seeing from towers in further steps (after Step $\tilde n$). If $n\neq n_t$ for any $t$, then $w_n=\tilde h_n$,
   if $n=n_t$ for some $t$, then $w_n=\tilde h_{n_t}+h_{l_t}$.
  \item $n_t,l_t$--- at step $n_t+1$ for $t\ge 1$, we add spacers while we do independent cutting and stacking,
  $h_{l_t}$ is the parameter related with the number of the spacers.
  \item $c_n$ ($=\# W_n=\#\tilde W_n$)--- the total number of the columns of $W_n$ or $\tilde W_n$.
  \item $\xi_n$ ($=\mu(|W_n|)=\mu(|\tilde W_n|)$)--- the total Lebesgue measure of the level sets in the tower $W_n$ or $\tilde{W}_n$.
  Recall that $\xi$ is the total length of the intervals $P_0$ and $P_1$, i.e. $\xi=\mu(|W_0|)$. In the following we will determine $\xi$ to make $\lim\limits_{n\rightarrow +\infty}{\xi}_{n}=1$. Since at each step $(n_t+1)$ we add spacers of measure $\xi_{n_t}\cdot \frac{h_{l_t}}{\tilde h_{n_t}}$,
the measures $\xi_n$'s of the tower $W_n$'s satisfy the following,
\begin{align}\label{equation-measure}
  &{\xi}_1={\xi}_2=\cdots={\xi}_{n_1}={\xi}, \nonumber \\
  &{\xi}_{n_t+1}={\xi}_{n_t}\frac{w_{n_t}}{\tilde h_{n_t}}={\xi}_{n_t}(1+\frac{h_{l_t}}{\tilde h_{n_t}}), \\
  &{\xi}_{n_t+1}={\xi}_{n_{t}+2}=\cdots={\xi}_{n_{t+1}}, t\ge 1. \nonumber
\end{align}
\end{itemize}
Due to the choice of $r_n$, $\sum_{t=1}^{+\infty}\frac{1}{r_{n_t}}$ converges. So $\sum_{t=1}^{+\infty}\frac{h_{l_t}}{\tilde h_{n_t}}<\sum_{t=1}^{+\infty}\frac{1}{r_{n_t}}$ converges.
Let
\begin{align}\label{xi}
  \xi=\prod_{t=1}^{+\infty}(1+\frac{h_{l_t}}{\tilde h_{n_t}})^{-1}.
\end{align}
Then we have $0<\xi<1$ and $\lim\limits_{n\rightarrow +\infty}{\xi}_{n}=1$.

We note that $\mu(X)=\lim\limits_{n\rightarrow +\infty}{\xi}_{n}=1$. And it is not hard to see from the construction that $\mu$ is $T$-invariant. We will show that
$\mu$ is in fact ergodic later in Remark \ref{ergodicity}.

\subsection{The upper bound of the entropy dimension.}\ \

For convenience, for a finite collection $\mathcal{A}$ consisting of measurable sets in $\mathcal{B}$ (need not to be a partition),
we denote $$H_{\mu}(\mathcal{A})=\sum_{A\in \mathcal{A}}-\mu(A)\log\mu(A)\text{ and } N_\mu (\mathcal{A})=\# \{ A\in \mathcal{A}: \mu(A)>0\}$$
For $\beta \in \mathcal{P}_X$ and $U\subseteq X$, denote by
$$\beta\cap U=\{B\cap U: B\in \beta\}.$$
Let $n, K\in \N$. We define
\begin{align*}
  U_{K,n}=\begin{cases}
|W_K|\setminus \Big(\bigcup_{i=1}^{n}T^{h_K-i}\big(base(W_K)\big)\Big), &\text { if }h_K>n;\\
\emptyset, &\text{ if }h_K\le n.
  \end{cases}
\end{align*}
Here we recall that $h_K$ is the heights of the tower $W_K$. Then we have the following
estimation.

\begin{lemma}\label{name-estimate} Given $k\in\N$, let $E$ be a level set of a column in $W_k$ and let $\alpha=\{E,X\setminus E\}$.
Then for any $\epsilon>0$, there exists a constant $M=M(\epsilon)>0$ such that when $n$ is sufficiently large,
\begin{align}\label{name}
 N_\mu(\bigvee_{i=0}^{n-1}T^{-i}\alpha\bigcap U_{K,n})\le (n^3+2n^2+n)2^{Mn^{\tau+\epsilon}}
\end{align}
for any $K\in \mathbb{N}$.
\end{lemma}
\begin{proof}  Firstly, we are to define $\mathcal C(n,K)$ for given $n,K\in \mathbb{N}$.
Let $n,K\in \mathbb{N}$. There are two cases. The first case is $h_K\le n$. In this case  we put $\mathcal{C}(n,K)=0$ and then
$$N_\mu(\bigvee_{i=0}^{n-1}T^{-i}\alpha\bigcap U_{K,n})\le \mathcal{C}(n,K)$$
since $U_{K,n}=\emptyset$.

The second case is $h_K> n$. In this case, for each column $C=\{B_1,B_2,\cdots,B_{h_K}\}$
of $W_K$, we can associate $C$ an $\alpha$-name $b=b_1b_2\cdots b_{h_K}\in \{u,v\}^{h_K}$ by
\begin{align*}
  b_i=
  \begin{cases}
u, &\text{ if }B_i\subset E;\\
v, &\text{ if }B_i\subset (X\setminus E).
  \end{cases}
\end{align*}
Let $\tilde E\subset U_{K,n}$ be a level set of a column in the tower $W_K$ and let $d=d_0d_1\cdots d_{n-1}\in\{u,v\}^n$ be the $\alpha$-name of the segment
$S=\{\tilde E,T\tilde E,T^2\tilde E,\cdots,T^{n-1}\tilde E\}$ (inherited from the $\alpha$-name of the column of $W_K$ that contains $S$).
Note that for each $0\le i\le n-1$,
\begin{align*}
d_i=
  \begin{cases}
u, &\text{ if }T^i\tilde E\subset E;\\
v, &\text{ if }T^i\tilde E\subset (X\setminus E).
  \end{cases}
\end{align*}
In fact $d$ is a subword of length $n$ of $\alpha$-names of $W_K$-segments.

Denote by
$$\mathcal{C}(n,K)=\#\{d\in \{u,v\}^{n}: d \text{ is a subword of }\alpha\text{-names of }W_K \text{-segments}\}.$$
Since any element in the collection $\bigvee_{i=0}^{n-1}T^{-i}\alpha\bigcap U_{K,n}$ is a union of some level sets in $W_K$  ($\text{mod } \mu$),
we have that
\begin{align}\label{eq-cnk}
N_\mu(\bigvee_{i=0}^{n-1}T^{-i}\alpha\bigcap U_{K,n})\le \mathcal C(n,K).
\end{align}

In the following we will show that for any $\epsilon>0$, there exists a constant $M=M(\epsilon)>0$ such that when $n$ is sufficiently large,
$$\mathcal C(n,K)\le (n^3+2n^2+n)2^{Mn^{\tau+\epsilon}}$$
for any $K\in \mathbb{N}$. Thus combining this fact with \eqref{eq-cnk}, one has \eqref{name}.

Recall that $c_n=\# W_n=\#\tilde W_n$ is the total number of the columns of the tower $W_n$ or $\tilde W_n$. From our construction,
  \begin{align}\label{cn}
    c_{n+1}=\begin{cases}
(c_n)^{e_n}, &\text{ if }n\notin \{n_1,n_2,\cdots \};\\
(c_n)^{e_n+1}, &\text{ if }n\in \{n_1,n_2,\cdots \}.
    \end{cases}
  \end{align}
  Hence $c_n$ has the expression
  \begin{align}\label{cn2}
    c_n&=2^{\big (\prod\limits_{0\le i\le n-1, i\notin \{n_1,n_2,\cdots\}}e_i\big)\cdot\big(\prod\limits_{0\le i\le n-1, i\in \{n_1,n_2,\cdots\}}(e_i+1)\big)}\nonumber\\
    &=2^{\big(\prod_{i=0}^{n-1}e_i\big )\cdot\big(\prod\limits_{0\le i\le n-1, i\in \{n_1,n_2,\cdots\}}(1+\frac{1}{e_i})\big)}.
  \end{align}
\medskip
Given $n\gg h_k$ and $K\in \mathbb{N}$. We consider two cases separately.

\medskip
\noindent{\bf Case I.} {\it Suppose $\tilde h_\ell\le n<h_{\ell+1}$ for some $\ell\in \N$. }

\medskip
If $h_K\le n$, then $\mathcal{C}(n,K)=0$. Now we assume $h_K>n$. Then $K\ge \ell+1$.
Let $$S=\{\tilde E,T\tilde E,T^2\tilde E,\cdots,T^{n-1}\tilde E\}$$
be any segment of height $n$ of a column in $W_K$ and $d$ be the $\alpha$-name of $S$.
By our construction, any column of $W_K$ is stacked by a sequence of
$W_{\ell+1}$-segments. According to the positions of $S$, there are two subcases.

\medskip
{\it Case (I.1).  $S$ is completely contained in some $W_{\ell+1}$-segment.}

\medskip
In this case $S$ has the form $S=pS_1S_2\cdots S_mq$,
where $S_i$ is some $\tilde W_{\ell}$-segment for each $1\le i \le m$, $p$ is an ending part of some $\tilde W_{\ell}$-segment $S_0$,
$q$ is a beginning part of some $\tilde W_{\ell}$-segment $S_{m+1}$ and $m=\lfloor\frac{n}{w_{\ell}}\rfloor$ or $\lfloor\frac{n}{w_{\ell}}\rfloor-1$. We should note that
the $\tilde W_{\ell}$-segments here may contain the inserted spacers if $\ell=n_t$ for some $t$.
The segment $S$ in this case is determined by $S_0, S_1, \cdots, S_{m+1}$ and the length of $p$.

For each $0\le i \le m+1$,
if $\ell\neq n_t+1$ for every $t$, then there are no more than $c_{\ell}=(c_{\ell-1})^{e_{\ell-1}}$-many choices for the $\tilde W_{\ell}$-segment $S_i$; if $\ell=n_t+1$ for some $t$, there are no more than $c_{\ell}=(c_{\ell-1})^{e_{\ell-1}+1}$-many choices for $S_i$. $m\le\lfloor\frac{n}{w_{\ell}}\rfloor\le \lfloor\frac{h_{\ell+1}}{w_{\ell}}\rfloor=e_{\ell}$.
There are $w_{\ell}$-many choices for the length of $p$, which is no more than $n$.
Then the total number of the $\alpha$-names of such $S$'s is bounded by
$$n\big ((c_{\ell-1})^{e_{\ell-1}+1}\big )^{e_{\ell}+2}.$$

{\it Case (I.2). $S$ is not completely contained in any $W_{\ell+1}$-segment. }

\medskip
Then there are two subcases.
\begin{enumerate}
  \item [(I.2.a).] $S$ has overlaps with two $W_{t}$-segments for
  some $t\ge \ell+1$.
  We can finally deduce that $S$ has overlaps with two $W_{\ell+1}$-segments (there may exist spacers of later step between them).
  Then $S$ has the form
  $S=S_0s^rS_1$, where $S_0$ is an ending part of some $W_{\ell+1}$-segment, $S_1$ is a beginning part of some $W_{\ell+1}$-segment
  and $s^r$ is $r$-many spacers between the two $W_{\ell+1}$-segments. $S$ is determined by $S_0S_1$, $r$ and the height of $S_0$. By a similar discussion as in Case (I.1),
  the number of the $\alpha$-names of such $S_0S_1$ is bounded by $n\big ((c_{\ell-1})^{e_{\ell-1}+1}\big )^{e_{\ell}+2}$. And both $r$ and the height of $S_0$ have no more than $n$ choices.
  Hence the total number of $d$'s in this subcase is bounded by $n^3\big ((c_{\ell-1})^{e_{\ell-1}+1}\big )^{e_{\ell}+2}$.

  \item [(I.2.b).] $S$ begins with some spacers and then followed by a beginning part of some $W_{\ell+1}$-segment or
  $S$ begins with an ending part of some $W_{\ell+1}$-segment and then followed by some spacers. Then $S$ has the form $S=s^rS'$ or $S=S's^r$, where $S'$
  is a segment which is completely contained in some $W_{\ell+1}$-segment.
  By a similar discussion as in Case (I.1), the $\alpha$-name of $S'$ has no more than $n\big ((c_{\ell-1})^{e_{\ell-1}+1}\big )^{e_{\ell}+2}$-many choices and
  $r$ has no more than $n$ choices.
  The total number of $d$'s in this subcase is bounded by $2n^2\big ((c_{\ell-1})^{e_{\ell-1}+1}\big )^{e_{\ell}+2}$.
\end{enumerate}

Summing the above estimations up, we have $$\mathcal C(n,K)\le (n^3+2n^2+n)\big ((c_{\ell-1})^{e_{\ell-1}+1}\big )^{e_{\ell}+2}.$$

Next by \eqref{cn2}, we have
\begin{align*}
  \big ((c_{\ell-1})^{e_{\ell-1}+1}\big )^{e_{\ell}+2}=2^{\bigg (\prod_{i=0}^{\ell}e_i\bigg )\cdot\bigg (\prod\limits_{0\le i\le \ell-2, i\in \{n_1,n_2,\cdots\}}(1+\frac{1}{e_i})\bigg )\cdot (1+\frac{1}{e_{\ell-1}})\cdot(1+\frac{2}{e_{\ell}})}.
\end{align*}
By \eqref{nt} and \eqref{condition2}, the product $\prod\limits_{i\in \{n_1,n_2,\cdots\}}(1+\frac{1}{e_i})$ is bounded.
By the definition of $w_{\ell}$ (see \eqref{condition0hw}), we have $w_{\ell}\le 2 \tilde h_{\ell}\le 2n$.
Hence
\begin{align}\label{prod}
  \prod_{i=0}^{\ell}e_i&\le (w_{\ell})^{\tau}\cdot e_{\ell} \text{ (by \eqref{condition3})}\nonumber \\
  &< 2n^{\tau}\cdot M'(r_{\ell})^{\frac{\tau}{1-\tau}}\text { (by \eqref{condition2})},
\end{align}
where $M'>0$ is a constant independent on $n$.
By \eqref{condition4}, for any given $\epsilon>0$, when $n$ is sufficiently large
(hence so is $\ell$),
$$(r_{\ell})^{\frac{\tau}{1-\tau}}\le (w_{\ell})^\epsilon\le 2n^{\epsilon}.$$
Hence for any $\epsilon>0$ we can find a constant $M=M(\epsilon)>0$, such that
$$\mathcal C(n,K)\le (n^3+2n^2+n)2^{Mn^{\tau+\epsilon}}$$
when $n$ is sufficiently large.

\medskip
\noindent{\bf Case II.} Suppose $h_\ell\le n<\tilde h_{\ell}$ for some $\ell\in \N$.

\medskip
Similar as in Case I, we assume $h_K>n$ and then $K\ge \ell+1$.
Let $$S=\{\tilde E,T\tilde E,T^2\tilde E,\cdots,T^{n-1}\tilde E\}$$
be any segment of height $n$ of a column in $W_K$ and $d$ be the $\alpha$-name of $S$.
By our construction, in this case, any column of $W_K$ is stacked by a sequence of
$\tilde W_{\ell}$-segments. Similar as in Case I, according to the positions of $S$, there are two subcases.

\medskip
{\it Case (II.1). $S$ is completely contained in some $\tilde W_{\ell}$-segment.}

 \medskip
Since $\tilde W_{\ell}$ is obtained by the repetitions of columns of $W_{\ell}$, in this case $S$ has the form $S=pS_0S_0\cdots S_0q$,
where $S_0$ is some $W_{\ell}$-segment, $p$ is an ending part of $S_0$ and
$q$ is a beginning part of $S_0$.
The segment $S$ in this case is determined by $S_0$ and the height of $p$.
There are $c_{\ell}$-many choices for $S_0$ and at most $n$ choices for the height of $p$.
Then the total number of the $\alpha$-names of such $S$'s is no more than $nc_{\ell}$.

\medskip
{\it Case (II.2). $S$ is not completely contained in any $\tilde W_{\ell}$-segment.}

 \medskip
Similar as in Case (I.2), there are again
two subcases.
\begin{enumerate}

  \item [(II.2.a).] $S$ has overlaps with two $\tilde W_{t}$-segments for
  some $t\ge \ell$.
  We can finally deduce that $S$ has overlaps with two $\tilde W_{\ell}$-segments (there may exist spacers of later step between them).
  Then $S$ has the form
  $S=S_0s^rS_1$, where $S_0$ is an ending part of some $\tilde W_{\ell}$-segment, $S_1$ is a beginning part of some $\tilde W_{\ell}$-segment
  and $s^r$ is $r$-many spacers between the two $\tilde W_{\ell}$-segments. $S$ is determined by $S_0$, $S_1$ and $r$.
  There are no more than $nc_{\ell}$-many choices for $S_0$ and $S_1$ and no more than $n$-many choices for $r$.
  Hence the total number of $d$'s in this subcase is bounded by $n^3(c_{\ell})^2$.

  \item [(II.2.b).] $S$ begins with some spacers and then followed by a beginning part of some $\tilde W_{\ell}$-segment or
  $S$ begins with an ending part of some $\tilde W_{\ell}$-segment and then followed by some spacers. Then $S$ has the form
  $S=s^rS'$ or $S=S's^r$, where $S'$ is a segment which is completely contained in some $\tilde W_{\ell}$-segment.
  By a similar discussion as in Case (II.1), the $\alpha$-name of $S'$ has no more than $nc_{\ell}$-many choices.
  $r$ has no more than $n$ choices.
  The total number of $d$'s in this subcase is bounded by $2n^2c_{\ell}$.
\end{enumerate}

Summing the above estimations up, we have in this case
$$\mathcal{C}(n,K)\le (n^3+2n^2+n)(c_{\ell})^2.$$

By \eqref{cn2}, we have
\begin{align*}
  c_{\ell}=2^{\bigg(\prod_{i=0}^{\ell-1}e_i\bigg)\cdot\bigg(\prod\limits_{0\le i\le \ell-1, i\in \{n_1,n_2,\cdots\}}(1+\frac{1}{e_i})\bigg)}.
\end{align*}
Noticing that in this situation, $w_{\ell-1}\le 2 \tilde h_{\ell-1}<2h_{\ell}\le 2n$,
we have
\begin{align*}
  \prod_{i=0}^{\ell-1}e_i&\le (w_{\ell-1})^{\tau}\cdot e_{\ell-1} \text{ (by \eqref{condition3})}\\
  &< 2n^{\tau}\cdot M'(r_{\ell-1})^{\frac{\tau}{1-\tau}}\text { (by \eqref{condition2})},
\end{align*}
where $M'>0$ is the same constant as appeared in \eqref{prod}.
By \eqref{condition4} again, for any given $\epsilon>0$, when $n$ is sufficiently large
(hence so is $\ell$),
$$(r_{\ell-1})^{\frac{\tau}{1-\tau}}\le (w_{\ell-1})^\epsilon\le 2n^{\epsilon}.$$
Similar as shown in Case I, for any $\epsilon>0$ we can find a constant $M=M(\epsilon)>0$, which depends on $\epsilon$ but is independent on $n$, such that
$$\mathcal{C}(n,K)\le (n^3+2n^2+n)2^{Mn^{\tau+\epsilon}}$$
when $n$ is sufficiently large. This finishes the proof of the lemma.
\end{proof}

With the help of Lemma \ref{name-estimate}, we are able to show that $\overline{D}_\mu(X,T)\le \tau$ as the following Lemma \ref{prop-upper}.
Hence for any partition $\beta=\{ B,X\setminus B\}\in \cal P^2_X$ with $0<\mu(B)<1$, we have $\overline{D}_\mu(T,\beta)\le \overline{D}_\mu(X,T)\le \tau$.
\begin{lemma}\label{prop-upper}
$\overline{D}_\mu(X,T)\le \tau$.
\end{lemma}
\begin{proof} Since
\begin{equation*}
\bigvee_{k=1}^{+\infty} \Big(\bigvee_{E\text{ is a level set of }W_k}\{E,X\setminus E\}\Big)=\mathcal{B}\ \ (\text{mod}\, \mu),
\end{equation*}
by (3) of Theorem \ref{thm-mart},
$$\overline{D}_\mu(X,T)=\sup_{E\text{ is a level set of some }W_k}\overline{D}_\mu(T,\{E,X\setminus E\}).$$
Hence it is sufficient to show that $\overline{D}_\mu(T,\alpha)\le \tau$
for any $\alpha=\{E,X\setminus E\}$, where $E$ is a level set of $W_k$ for some $k\in \mathbb{N}$.

Given $k\in\N$, let $E$ be a level set in $W_k$ and let $\alpha=\{E,X\setminus E\}$.
In the following we are to show that $\underline{D}(S)\le \tau$
for any $S\in \mathcal{P}_\mu(T,\alpha)$, which implies $\overline{D}_\mu(T,\alpha)\le \tau$ by Definition \ref{de-s-3}.

If this is not true, then we can find  $S=\{s_1<s_2<\cdots\}\in \mathcal{P}_\mu(T,\alpha)$  and $\epsilon>0$ such that $\underline{D}(S)>\tau+\epsilon$.
It is clear that $\liminf\limits_{n\rightarrow+\infty}\frac{n}{(s_n)^{\tau+\epsilon}}=+\infty$.

By Lemma  \ref{name-estimate},  there exists a constant $M=M(\epsilon)>0$  and $N_\epsilon\in \mathbb{N}$
such that
\begin{align}\label{name-1}
 N_\mu(\bigvee_{i=0}^{n-1}T^{-i}\alpha\bigcap U_{K,n})\le (n^3+2n^2+n)2^{Mn^{\tau+\epsilon}}
\end{align}
for any $K\in \mathbb{N}$ and $n\ge N_\epsilon$.

Now for each $n\ge N_\epsilon$, since $\mu(U_{K,n})=\xi_K\cdot(1-\frac{n}{h_K})$ when $h_K>n$, we have
$$\mu(X\setminus U_{K,n})=(1-\xi_K)+\frac{n}{h_K}\xi_K.$$
Hence we can choose $K=K(n)$ sufficiently large to satisfy that
$$\mu(X\setminus U_n)\le \frac{1}{2n\log 2}$$
and
$$-\mu(U_n)\log\mu(U_n)-\mu(X\setminus U_n)\log \mu(X\setminus U_n)\le\frac{1}{2},$$
where for simplicity we write $U_n$ as $U_{K(n),n}$.

Together with the fact $\#(\bigvee_{i=0}^{n-1}T^{-i}\alpha)\le 2^n$, we have
\begin{align*}
  \mu(X\setminus U_n)\log \#(\bigvee_{i=0}^{n-1}T^{-i}\alpha)-\mu(U_n)\log\mu(U_n)-\mu(X\setminus U_n)\log \mu(X\setminus U_n)\le 1.
\end{align*}
Hence when $n\ge N_\epsilon$,
\begin{align*}
 &H_{\mu}(\bigvee_{i=0}^{n-1}T^{-i}\alpha)\nonumber \\
 \le& H_{\mu}\Big(\bigvee_{i=0}^{n-1}T^{-i}\alpha\bigvee\{U_n,X\setminus U_n\}\Big) \nonumber \\
 =&H_{\mu}\Big(\bigvee_{i=0}^{n-1}T^{-i}\alpha\bigcap U_n\Big)+H_{\mu}\Big(\bigvee_{i=0}^{n-1}T^{-i}\alpha\bigcap (X\setminus U_n)\Big) \nonumber \\
 \le&-\mu(U_n)\log \frac{\mu(U_n)}{N_\mu(\bigvee_{i=0}^{n-1}T^{-i}\alpha\bigcap U_n)}
 -\mu(X\setminus U_n)\log \frac{\mu(X\setminus U_n)}{N_\mu(\bigvee_{i=0}^{n-1}T^{-i}\alpha)} \nonumber \\
 =&\mu(U_n)\log N_\mu(\bigvee_{i=0}^{n-1}T^{-i}\alpha\bigcap U_n)+\mu(X\setminus U_n)\log \#(\bigvee_{i=0}^{n-1}T^{-i}\alpha) \nonumber \\
 &\ \ \ \ \ \ -\mu(U_n)\log\mu(U_n)-\mu(X\setminus U_n)\log \mu(X\setminus U_n) \nonumber \\
 \le& \mu(U_n)\log N_\mu(\bigvee_{i=0}^{n-1}T^{-i}\alpha\bigcap U_n)+1\nonumber \\
 \le& \log \Big((n^3+2n^2+n)2^{Mn^{\tau+\epsilon}}\Big)+1,
\end{align*}
where the last inequality comes from \eqref{name-1}.

Now using the above estimation we have
\begin{align*}
 h^S_{\mu}(T,\alpha)&=\limsup_{n\rightarrow+\infty}\frac{1}{n}H_{\mu}(\bigvee_{i=1}^nT^{-s_i}\alpha)\\
 &\le \limsup_{n\rightarrow+\infty}\frac{1}{n}H_{\mu}(\bigvee_{i=0}^{s_n}T^{-i}\alpha) \\
 &\le \limsup_{n\rightarrow+\infty}\frac{1}{n}\cdot \Big(\log \big(((s_n+1)^3+2(s_n+1)^2+s_n+1)2^{M(s_n+1)^{\tau+\epsilon}}\big)+1\Big)\\
 &\le \limsup_{n\rightarrow+\infty}\frac{1}{n}\cdot (s_n+1)^{\tau+\epsilon}\cdot M\log 2=0.
\end{align*}
That is, $h^S_{\mu}(T,\alpha)=0$, a contradiction with $S\in \mathcal{P}_\mu(T,\alpha)$.
This implies $\overline{D}_\mu(T,\alpha)\le \tau$.
\end{proof}

\subsection{The lower bound.}\ \

For $A,B\subset\Z$, let $A+B\triangleq\{a+b: a\in A,b\in B\}$ and let $|A|$ denote the number of integers in $A$.
Recall that $w_n$ is given in \eqref{condition0hw}. For $t\ge 1$, let
\begin{align}\label{def-fl}
&F_0^t=\{0,w_{n_t},2w_{n_t},\cdots,(e_{n_t}-1)w_{n_t}\}, \nonumber\\
&F_{k}^t=F_{k-1}^t+\{0,w_{n_t+k},2w_{n_t+k},\cdots,(e_{n_t+k}-1)w_{n_t+k}\} \text{ for } k\ge 1,\text{ and }\\
&F^t=\bigcup \limits_{k=0}^{+\infty} {F}_k^t\nonumber.
\end{align}
We have $F_0^t\subset F_1^t\subset \cdots$ and $|{F}_k^t|=e_{n_{t}}e_{n_{t}+1}\cdots e_{n_{t}+k}.$
\begin{lemma}\label{prop-F}
For any $t\ge 1$, $D(F^t)=\tau$.
\end{lemma}
\begin{proof}
Given $t\ge 1$, let $F^t=\{t_1<t_2<\cdots\}$. For any $n\in\N$, there exists a unique $k=k(n)$ such that $t_n\in F_{k+1}^t\setminus F_k^t$.
Then $$e_{n_{t}}e_{n_{t}+1}\cdots e_{n_{t}+k}< n\le e_{n_{t}}e_{n_{t}+1}\cdots e_{n_t+k+1}$$
and $$h_{n_t+k+1}< t_n\le h_{n_t+k+2}.$$

For any $\tau'$ with $0\le\tau' < \tau,$
\begin{align} \label{es-key1}
\underline{D}(F^t,\tau')&=\liminf_{n\rightarrow+\infty}\frac{n}{(t_n)^{\tau'}}\nonumber\\
&\ge \liminf_{k\rightarrow+\infty}\frac{e_{n_{t}}e_{n_{t}+1}\cdots e_{n_{t}+k}}{(h_{n_t+k+2})^{\tau'}}
=\liminf_{k\rightarrow+\infty}\frac{e_{n_{t}}e_{n_{t}+1}\cdots e_{n_{t}+k}}{(w_{n_{t}+k+1}e_{n_{t}+k+1})^{\tau'}} \nonumber\\
&=\liminf_{k\rightarrow+\infty}\frac{e_0e_1\cdots e_{n_{t}+k+1}}{(w_{n_{t}+k+1}e_{n_{t}+k+1})^{\tau}}\cdot
\frac{(w_{n_{t}+k+1}e_{n_{t}+k+1})^{\tau-\tau'}}
{e_0e_1\cdots e_{n_t-1}e_{n_{t}+k+1}}\nonumber\\
&=+\infty.
\end{align}
We note that the last equality comes from \eqref{condition1} and \eqref{condition4}.
Hence $\underline D(F^t)\ge\tau'$. Since this inequality is true
for any $\tau'\in [0,\tau)$, we have $\underline{D}(F)\ge \tau$.

For any $\tau'$ with $\tau< \tau'<1,$
\begin{align} \label{es-key2}
\overline{D}(F^t,\tau')&=\limsup_{n\rightarrow+\infty}\frac{n}{(t_n)^{\tau'}}\nonumber\\
&\le \limsup_{k\rightarrow+\infty}\frac{e_{n_{t}}e_{n_{t}+1}\cdots e_{n_t+k+1}}{(h_{n_t+k+1})^{\tau'}}
=\limsup_{k\rightarrow+\infty}\frac{e_{n_{t}}e_{n_{t}+1}\cdots e_{n_t+k+1}}{(w_{n_t+k}e_{n_t+k})^{\tau'}} \nonumber\\
&=\limsup_{k\rightarrow+\infty}\frac{e_0e_1\cdots e_{n_t+k}}{(w_{n_t+k}e_{n_t+k})^{\tau}}\cdot
\frac{e_{n_t+k+1}}
{e_0e_1\cdots e_{n_t-1}\cdot(w_{n_t+k}e_{n_t+k})^{\tau'-\tau}} \nonumber\\
&=0,
\end{align}
where the last equality comes again from \eqref{condition1} and \eqref{condition4}.
So $\overline D(F^t)\le\tau'$. Since this inequality is true
for any $\tau'\in(\tau,1)$, we have $\overline{D}(F^t)\le \tau$. Hence $D(F^t)=\tau$.
\end{proof}

\begin{lemma}\label{lemma-independent1}
Given $t>0$ and $k\ge0$, let $B\subset {F}_k^t$ and $E_b$ be a level set in $W_{l_{t}}$ for $b\in B$ ($E_b$'s need not to have different names), then
\begin{align}\label{equation-independent1}
&\mu(\bigcap_{b\in B}T^{-b}E_b)\le(1+\frac{h_{l_{t}}}{c_{n_{t}}})^{|B|}\cdot(\frac{1}{\xi_{l_t}})^{|B|}\prod_{b\in B}\mu(E_b).
\end{align}
Moreover, let $U_b$ be a union of finite many level sets in $W_{l_t}$ for $b\in B$, then
\begin{align}\label{equation-independent2}
&\mu(\bigcap_{b\in B}T^{-b}U_b)\le(1+\frac{h_{l_{t}}}{c_{n_{t}}})^{|B|}\cdot(\frac{1}{\xi_{l_t}})^{|B|}\prod_{b\in B}\mu(U_b).
\end{align}
\end{lemma}
\begin{proof}
Assume $B=\{b_1,b_2,\cdots,b_m\}$ and $E_{b_i}$'s ($i=1,2,\cdots,m$) be level sets in $W_{l_{t}}$. Then every $E_{b_i}$ satisfies that
$$\mu(E_{b_i})=\frac{\xi_{l_t}}{c_{l_t}\cdot h_{l_t}}.$$

Notice that the level sets $E_{b_i}$'s in $W_{l_{t}}$ are all spread out into many much smaller level sets after
sufficiently large steps. For the small level sets $A_1$'s from $E_{b_1}$, to ensure the level sets $T^{b_i-b_1}A_1$'s ($2\le i \le m$) are from $E_{b_i}$ respectively,
the $W_{l_{t}}-$segment which contains $T^{b_i-b_1}A_1$ must be isomorphic with the column that contains $E_{b_i}$ in $W_{l_{t}}$ for every $2\le i \le m$.
This situation happens with probability $\big (\frac{1}{c_{l_t}}\big )^{m-1}$.
Also we need for each $2\le i \le m$, the positions of $E_{b_i}$'s in $W_{l_{t}}$ coincide with the positions of $E_{b_i}$'s in $W_{l_{t}}-$segments after inserting spacers. Since the numbers of beginning spacers of $W_{l_{t}}-$segments
are uniformly distributed on $\{0,1,\cdots,h_{l_t}-1\}$ within $\frac{h_{l_t}}{c_{n_t}}-$error, at most $(\frac{1}{h_{l_t}}+\frac{1}{c_{n_t}})^{m-1}$-portion of them coincide. So
\begin{align*}
\mu\big(\bigcap_{b\in B}T^{-b}E_b\big)&=\mu\big(T^{-b_1}E_{b_1}\bigcap T^{-b_2}E_{b_2}\bigcap \cdots \bigcap T^{-b_m}E_{b_m}\big)\\
&\le\mu(E_{b_1})\cdot \big (\frac{1}{c_{l_t}}\big )^{m-1}\cdot \big(\frac{1}{h_{l_t}}+\frac{1}{c_{n_t}}\big)^{m-1} \\
&=\big(1+\frac{h_{l_{t}}}{c_{n_{t}}}\big)^{m-1}\cdot\big(\frac{1}{\xi_{l_t}}\big)^{m-1}\mu(E_{b_1})\mu(E_{b_2})\cdots \mu(E_{b_m})\\
&\le\big(1+\frac{h_{l_{t}}}{c_{n_{t}}}\big)^{|B|}\cdot\big(\frac{1}{\xi_{l_t}}\big)^{|B|}\prod_{b\in B}\mu(E_b).
\end{align*}

Since $U_b$ is a disjoint union of level sets in $W_{l_t}$, inequality \eqref{equation-independent2} then follows from \eqref{equation-independent1}.
\end{proof}

\begin{remark}\label{ergodicity}
From the above lemma, for any $p\in \Z^+$ and any two level sets $E$ and $\tilde E$, there exists $n>0$ with $\mu(T^{-np}E\cap \tilde{E})>0$.
We note that the $\sigma-$algebra $\mathcal{B}$ is generated by the level sets. Approximated by the union of these level sets, for any two sets $A$ and $\tilde{A}$
with positive measures, there also exists $n>0$ with $\mu(T^{-np}A\cap \tilde{A})>0$. This implies that $\mu$ is an ergodic measure under $T^p$ for any $p$.
\end{remark}

In the following we will prove the u.d. property for the partition $\{A,A^c\}$, where $A$ is a union of finite many
level sets in $W_{\ell}$ for some $\ell$.
\begin{lemma}\label{lemma-independent2}
Let $\alpha=\{A,A^c\}$, where $A$ is a union of finite level sets in $W_{\ell}$ for some $\ell\in\N$ with $0<\mu(A)\le \frac{1}{2}\xi_{\ell}$.
Then for sufficiently large $t$,
$$\liminf_{n\rightarrow+\infty}\frac{1}{n}H_{\mu}(\bigvee_{i=1}^n
T^{-t_i}\alpha)\ge-\frac{1}{2}\mu(A)\log\frac{\mu(A)}{1-\mu(A)}>0,$$
where $F^t=\{t_1<t_2<\cdots\}$ is given by \eqref{def-fl}.
Hence $F^t$ is an entropy generating sequence of $\alpha$. Moreover, $D_{\mu}(T, \alpha)=\tau$.
\end{lemma}
\begin{proof}
Since $0<\mu(A)\le \frac{1}{2}\xi_{\ell}<\frac{1}{2}$, we have $-\mu(A)\log\frac{\mu(A)}{1-\mu(A)}>0$.
Note that $l_k<n_k$ and $e_k\ge 2$ for any $k\in \mathbb{N}$. By \eqref{condition0hw},\eqref{condition1} and \eqref{cn2}, one has $\lim\limits_{t\rightarrow +\infty} \frac{h_{l_{t}}}{c_{n_{t}}}=0$. Thus combining this fact with $\lim\limits_{n\rightarrow +\infty} \xi_{n}=1$, we can take $t$ sufficiently large such that $l_t\ge\ell$ and
$$\log\Big((1+\frac{h_{l_{t}}}{c_{n_{t}}})\cdot\frac{1}{\xi_{l_t}}\Big)<-\frac{1}{2}\mu(A)\log\frac{\mu(A)}{1-\mu(A)}.$$
For convenience, let $A_0=A,A_1=A^c$.
For any finite subset $B$ of $F^t$ and any finite sequence $s=(s_b)_{b\in B}\in \{0,1\}^{B}$, let
$$B_0(s)=\{b\in B:s_b=0\}\text{ and }B_1(s)=\{b\in B:s_b=1\}.$$
Noticing that $A_0=A$ is a union of finite many level sets in $W_{l_t}$ (we note here that since $A_1=A^c$ is not a union of finite many level sets in $W_{l_t}$,
we can not apply Lemma \ref{lemma-independent1} to $\mu(\bigcap_{b\in B}T^{-b}A_{s_b})$), we have
\begin{align*}
&\mu(\bigcap_{b\in B}T^{-b}A_{s_b})\le \mu(\bigcap_{b\in B_0(s)}T^{-b}A_0)\\
\le&(1+\frac{h_{l_{t}}}{c_{n_{t}}})^{|B_0(s)|}\cdot(\frac{1}{\xi_{l_t}})^{|B_0(s)|}\prod_{b\in B_0(s)}\mu(A_0)\text{ (by Lemma \ref{lemma-independent1}) }\nonumber \\
=&\big(\prod_{b\in B_0(s)}\mu(A_0)\big)\cdot \big(\prod_{b\in B_1(s)}\mu(A_1)\big)\cdot (1+\frac{h_{l_{t}}}{c_{n_{t}}})^{|B_0(s)|}\cdot(\frac{1}{\xi_{l_t}})^{|B_0(s)|}\cdot (\frac{1}{\mu(A_1)})^{|B_1(s)|}\nonumber \\
\le& \big(\prod_{b\in B}\mu(A_{s_b})\big)\cdot (1+\frac{h_{l_{t}}}{c_{n_{t}}})^{|B|}\cdot(\frac{1}{\xi_{l_t}})^{|B|}\cdot (\frac{1}{\mu(A_1)})^{|B|}\nonumber.
\end{align*}

So
\begin{align*}
 &H_{\mu}(\bigvee_{i=1}^{m}T^{-t_i}\alpha)=\sum_{s\in\{0,1\}^m}-\mu(\bigcap_{i=1}^mT^{-t_i}A_{s_i})\log \Big(\mu(\bigcap_{i=1}^mT^{-t_i}A_{s_i})\Big) \\
 \ge& \sum_{s\in\{0,1\}^m}-\mu(\bigcap_{i=1}^mT^{-t_i}A_{s_i})\log\Big(\big(\prod_{i=1}^m\mu(A_{s_i})\big)\cdot (1+\frac{h_{l_{t}}}{c_{n_{t}}})^{m}\cdot(\frac{1}{\xi_{l_t}})^{m}\cdot (\frac{1}{\mu(A_1)})^{m}\Big)\\
 =&\Big(\sum_{s\in\{0,1\}^m}-\mu(\bigcap_{i=1}^mT^{-t_i}A_{s_i})\log\big(\prod_{i=1}^m\mu(A_{s_i})\big)\Big)-
 \log\Big((1+\frac{h_{l_{t}}}{c_{n_{t}}})^{m}\cdot(\frac{1}{\xi_{l_t}})^{m}\cdot (\frac{1}{\mu(A_1)})^{m}\Big).
\end{align*}
Since
\begin{align*}
 &\sum_{s\in\{0,1\}^m}-\mu(\bigcap_{i=1}^mT^{-t_i}A_{s_i})\log\Big(\prod_{i=1}^m\mu(A_{s_i})\Big)\\
 =&\sum_{j=1}^m\sum_{s\in\{0,1\}^m}-\mu(\bigcap_{i=1}^mT^{-t_i}A_{s_i})\log\Big(\mu(A_{s_j})\Big) \\
 =&\sum_{j=1}^m\sum_{s_j\in\{0,1\}}-\mu(T^{-t_j}A_{s_j})\log\Big(\mu(A_{s_j})\Big)\\
 =&mH_{\mu}(\alpha),
\end{align*}
we have
\begin{align}\label{equation-entropy1}
 H_{\mu}(\bigvee_{i=1}^{m}T^{-t_i}\alpha)&\ge mH_{\mu}(\alpha)-\log\Big((1+\frac{h_{l_{t}}}{c_{n_{t}}})^{m}\cdot(\frac{1}{\xi_{l_t}})^{m}\cdot (\frac{1}{\mu(A_1)})^{m}\Big)\nonumber \\
 &> m\Big(H_{\mu}(\alpha)+\frac{1}{2}\mu(A)\log\frac{\mu(A)}{1-\mu(A)}+\log\big(1-\mu(A)\big)\Big) \nonumber \\
 &= m\Big(-\frac{1}{2}\mu(A)\log\frac{\mu(A)}{1-\mu(A)}\Big).
\end{align}

Hence
$$\liminf_{n\rightarrow+\infty}\frac{1}{n}H_{\mu}(\bigvee_{i=1}^n
T^{-t_i}\alpha)\ge-\frac{1}{2}\mu(A)\log\frac{\mu(A)}{1-\mu(A)}>0,$$
which implies that $F^t$ is an entropy generating sequence of $\alpha$.
By Lemma \ref{prop-upper} and \ref{prop-F}, we have
$D_{\mu}(T, \alpha)=\tau$.
\end{proof}

\begin{theorem} $(X,\mathcal{B},\mu,T)$ is a
$\tau$-u.d. system.
\end{theorem}
\begin{proof} Let $\beta=\{B,X\setminus B\}\in \cal P^2_X$. We first consider the case for
$0<\mu(B)<\frac{1}{2}$. Then $$c(\beta):=-\frac{1}{2}\mu(B)\log\frac{\mu(B)}{1-\mu(B)}>0.$$ For any
$0<\epsilon<\frac{1}{2}c(\beta)$,
 by Lemma 4.15 of
\cite{Wal}, we can choose $\delta>0$ small enough such that
$H_\mu(\beta|\gamma)+H_\mu(\gamma|\beta)<\epsilon$, whenever
$\gamma=\{E,X\setminus E\}\in \mathcal{P}^2_X$ satisfies that $\mu(B\Delta
E)+\mu\big((X\setminus B)\Delta (X\setminus E)\big)<\delta$. When $\ell$ is sufficiently large, there is a
subset $A$ of $X$ which is a union of level sets in $W_{\ell}$
such that $0<\mu(A)\le \frac{1}{2}\xi_{\ell}$ and
$$\mu(A\Delta B)<\frac{\delta}{2}, \mu\big((X\setminus A)\Delta (X\setminus B)\big)<\frac{\delta}{2}.$$ Let
$\alpha=\{A,X\setminus A\}$, then
$H_\mu(\alpha|\beta)+H_\mu(\beta|\alpha)<\epsilon$. Moreover, we can make $c(\alpha)>\frac{1}{2}c(\beta)$ when $\delta$ is sufficiently small,
where $c(\alpha)=-\frac{1}{2}\mu(A)\log\frac{\mu(A)}{1-\mu(A)}>0.$
By Lemma \ref{prop-F} and Lemma \ref{lemma-independent2}, there exists $F=\{t_1<t_2<\cdots\}\subseteq \mathbb{N}$ such that $D(F)=\tau$ and
\begin{align*}
\liminf_{n\rightarrow+\infty}\frac{1}{n}H_{\mu}(\bigvee_{i=1}^n
T^{-t_i}\alpha) \ge c(\alpha).
\end{align*}
Thus we have
\begin{align*}
&\hskip0.5cm \liminf_{n\rightarrow +\infty}\frac{1}{n}H_\mu(\bigvee_{i=1}^n T^{-t_i}\beta) \\
&=\liminf_{n\rightarrow +\infty} \frac{1}{n}\left(H_\mu(\bigvee_{i=1}^n T^{-t_i}(\alpha\vee \beta))
-H_\mu(\bigvee_{i=1}^n T^{-t_i}\alpha|\bigvee_{i=1}^n T^{-t_i}\beta)\right)\\
&\ge \liminf_{n\rightarrow +\infty} \frac{1}{n}\left(
H_\mu(\bigvee_{i=1}^n T^{-t_i}\alpha )-nH_\mu(\alpha|\beta) \right )\\
&\ge c(\alpha)-H_\mu(\alpha|\beta)\ge c(\alpha)-\epsilon \\
&\ge c(\alpha)-\frac{1}{2}c(\beta)>0,
\end{align*}
which means that $F$ is also an entropy generating sequence of $\beta$. Hence $\underline{D}_{\mu}(T, \beta)\ge
\underline{D}(F)\ge\tau$. Combining with Lemma
\ref{prop-upper}, $D_{\mu}(T, \beta)=\tau$.

Next we consider the case for $\mu(B)=\frac{1}{2}$. Assume $\underline{D}_{\mu}(T, \beta)<\tau$.
Since $\mu$ is ergodic under both $T$ and $T^2$ (Remark \ref{ergodicity}), we have $\mu(B\cap T^{-1}B)\neq 0, \frac{1}{2}$.
Thus either $0<\mu(B\cap T^{-1}B)<\frac{1}{2}$ or $0<\mu(X\setminus (B\cap T^{-1}B))<\frac{1}{2}$. From the previous case,
we have $D_{\mu}\Big(T, \{B\cap T^{-1}B, X\setminus (B\cap T^{-1}B)\}\Big)=\tau$.
Noticing that $\{B\cap T^{-1}B, X\setminus (B\cap T^{-1}B)\}\preceq \beta\bigvee T^{-1}\beta$, by (1) and (2) of Proposition \ref{prop-basic},
we have
\begin{align*}
&\underline{D}_{\mu}\Big(T, \{B\cap T^{-1}B, X\setminus (B\cap T^{-1}B)\}\Big)\\
\le&\underline{D}_{\mu}(T, \beta\bigvee T^{-1}\beta)\\
=&\underline{D}_{\mu}(T, \beta)<\tau,
\end{align*}
which leads a contradiction. Hence we still have $D_{\mu}(T, \beta)=\tau$.

Since $\beta$ is arbitrary, we conclude that $(X,\mathcal{B},\mu,T)$ is a $\tau-$u.d. system.
\end{proof}

\begin{remark}
By the similar method, we can also choose suitable parameters such
that $(X,\cal B,T,\mu)$ is a $1-$u.d. system with zero entropy.
\end{remark}

{\bf Acknowledgements}
The authors would like to express their gratitude to the referees
for their valuable suggestions and comments.
The first author is supported by NNSF of China (Grant No. 10901080, 11271191, 11790274) and he also would like to thank BK21 Program of Department of
Mathematics, Ajou University for the pleasant support while we had started this work. The second author is supported by NNSF of China (Grant No. 11225105 and 11431012). The third author is
supported in part by NRF 2010-0020946.

\end{document}